\documentclass[12pt]{amsart}
\usepackage{amsfonts}
\numberwithin{equation}{section}
\usepackage{amssymb}
\usepackage{amsmath}

\newtheorem{lem}{Lemma}[section]
\newtheorem{thm}{Theorem}[section]

\newtheorem{exa}{Example}[section]
\theoremstyle{definition}

\theoremstyle{remark}

\title[Differences]{Studies of Differences from the point of view of Nevanlinna Theory}
\author[Zheng and Korhonen]{ Zheng Jianhua and Risto Korhonen}
\thanks{{2000 Mathematics Subject Classification: Primary 39A10, Secondary 30D35, 39A12}\\
{The first author is partially supported by the grant (No. 11571193)
of NSF of
China.}\\
{The second author is supported in part by the Academy of Finland grant 286877}}
\address{Department of Mathematical Sciences, Tsinghua University, P. R.
China} \email{zheng-jh@mail.tsinghua.edu.cn}
\address{Department of Physics
and Mathematics, University of Eastern Finland, P. O. Box 111, 80101
Joensuu, Finland} \email{risto.korhonen@uef.fi}

\begin{document}

\allowdisplaybreaks

\begin{abstract} This paper consists of three parts. First, we give so far the best
condition under which the shift invariance of the counting function,
and of the characteristic of a subharmonic function, holds. Second,
a difference analogue of logarithmic derivative of a
$\delta$-subharmonic function is established allowing the case of
hyper-order equal to one and minimal hyper-type, which improves the
condition of the hyper-order less than one. Finally, we make a
careful discussion of a well-known difference equation and give out
the possible forms of the equation under a growth condition for the
solutions.
\end{abstract}

\maketitle

\section{Introduction and Main Results}\label{intr}

One of the main purposes of Nevanlinna theory in the study of
difference equations is to single out those equations which are
assumed to have meromorphic solutions of finite order and then to
make a further careful analysis of the singled-out difference
equations by local singularity analysis and other methods. From the
point of view of discrete Painlev\'e equations, much interest has
been attracted by the following algebraic difference equations
\begin{equation}\label{equ1.1++}\overline{w}=R(z,w),\end{equation}
\begin{equation}\label{equ1.2++}\overline{w}+\underline{w}=R(z,w),\end{equation}
\begin{equation}\label{equ1.3++}\overline{w}\underline{w}=R(z,w),\end{equation}
\begin{equation}\label{equ1.4++}(\overline{w}+w)(w+\underline{w})=R(z,w),
\end{equation}
 where, henceforth,
$\overline{w}=w(z+1), \underline{w}=w(z-1)$ and $R(z,w)$ is a
rational function in $z$ and $w$. The eye-catching reason is that
Nevanlinna theory has proved to be powerful in the study of these
kinds of difference equations, although they have widely been
considered by other methods. The basic method, when implementing
Nevanlinna theory to study the equations above, consists of two
steps: The first is that in view of the Valiron-Mohon'ko Theorem
(\cite{Valiron}\cite{Mohon'ko}), an estimate of the degree of
$R(z,w)$ in $w$ is obtained, in fact, ruling out the higher degree
equations. These type of results may be called difference Malmquist
type theorems. As we know, ${\rm deg}_w(R)\leq 1$ for
(\ref{equ1.1++}) (\cite{Yanagihara}); ${\rm deg}_w(R)\leq 2$ for
(\ref{equ1.2++}) and (\ref{equ1.3++}) (\cite{Ablowitz}); ${\rm
deg}_w(R)\leq 4$ for (\ref{equ1.4++}) (\cite{Ramani}) if those
difference equations are assumed to have a transcendental
meromorphic solution of finite order. The second step is that the
equations are classified into several special types by analyzing the
behavior near poles of a meromorphic solution. The basic idea here
is related to singularity confinement, but the analysis goes deeper.
This analysis reduces the above equations into linear equations, or
the Riccati equation with degree one, or to one of the difference
Painlev\'e equations (see, e.g., \cite{HalburdKorhonen}). To
determine high pole density of meromorphic solutions is an
application of a difference Clunie type theorem~(\cite{Clunie},
\cite{LaineYang} and \cite{Risto}).

A difference rational function is a fraction of two irreducible
difference polynomials and its total degree is the maximal one of
degrees of the numerator and the denominator. Let us consider the
following difference polynomial equation
\begin{equation}\label{equ1.1+}U(z,\vec{w})P(z,\vec{w})=Q(z,\vec{w}),
\end{equation}
where $U, P$ and $Q$ are three difference polynomials in
$\vec{w}(z)=(w(z),w(z+c_1),...,w(z+c_n))$ for non-zero complex
numbers $c_j$ and $UP$ does not have common factors with $Q$ with
respect to $\vec{w}$. We call the equation (\ref{equ1.1+}) Clunie
type equation after J. Clunie who is the first one to research such
a differential polynomial equations (\cite{Clunie}).

In view of the method above to study this topic, we first establish
a difference Malmquist type theorem, that is to say, find a
relationship among the degrees of three difference polynomials $U,
P$ and $Q$. According to the well-known methods in terms of
Nevanlinna theory, we should first establish a difference version of
the Valiron-Mohon'ko Theorem, in the other words, we are asked to
extend the Valiron-Mohon'ko Theorem to the case of difference
polynomials dealing with the shifts of the function in question.
Therefore, we are forced to consider when
\begin{equation}\label{equ1.2+}T(r,w(z+c))\sim T(r,w(z))\end{equation} as $r\to\infty$ on at least a
sequence of positive numbers tending to $+\infty$ for all finite
complex numbers $c$ where $T(r,w)$ denotes the Nevanlinna
characteristic of $w$. (For basic definitions of Nevanlinna theory, we refer to \cite{Hayman,Zheng}.) It is easily seen that (\ref{equ1.2+}) is
essentially fundamental in the sense that existence of a
transcendental meromorphic solution with (\ref{equ1.2+}) of
difference equations (\ref{equ1.1++}) -- (\ref{equ1.4++}) derives
easily the above-mentioned estimates of the degrees in $w$ of the
rational function $R(z,w)$. Generally, we can establish Theorem~\ref{thm1.1} (see Section~\ref{Clunie_sect} below) which shows the significance of (\ref{equ1.2+}).

Unfortunately, (\ref{equ1.2+}) does not generally hold. An explicit
example is $w(z)=e^{e^z}$ which is such that
$T(r,w(z+1))=m(r,w(z+1))=em(r,w(z))=eT(r,w)$ for all $r>0$. From the
point of view of difference equations, it is easy to find such an
example so that (\ref{equ1.2+}) does not hold. It is proved in
\cite{Shimomura} that for any non-constant polynomial $P(w)$, the
difference equation $w(z+1)=P(w(z))$ has a non-trivial entire
solution. If, in addition, the degree of $P$ satisfies ${\rm
deg}(P)>1$, then none of the sequences of positive numbers tending
to $\infty$ is such that (\ref{equ1.2+}) holds for the entire
solution $w$. A generalization is made in \cite{Yanagihara} with the
$P$ replaced by a rational function $R$, according to which
 the difference equation $w(z+1)=R(w(z))$ always has a
non-trivial meromorphic solution $w$ and for such a solution $w$, we
have $T(r,\overline{w})={\rm deg}_w(R)T(r,w)+O(1)$, and hence
$$T(r,w)\geq KD^r, \quad \forall\ r>0,$$
where $D$ is any positive number less than the degree ${\rm
deg}_w(R)$ of $R$ and $K$ is a positive constant depending on $D$.
Conversely, these facts give us a heads up that (\ref{equ1.2+}) may
be true if $\log T(r,w)=o(r) (r\to\infty)$.

Studying the possibility of (\ref{equ1.2+}) we may ascend to works
of many great mathematicians of the last century, such as Valiron,
Dug\'ue, Hayman, Goldberg and Ostrovskii, only mentioning some of
them here. The reason to study this topic is very natural, as Goldberg
and Ostrovskii pointed out in their monograph \cite{GoldbergeOstrovskii}. All of the functions
with the form $f(qz+c)$ map the complex plane onto the same Riemann
surface. It thus makes sense to study the invariance of quantities
characterizing the asymptotic behavior of $w(z)$ under the linear
transformation $qz+c$ which can be decomposed into $qz$ and $z+c$.
The change produced by $qz$ is clear, so the attention is put on the
shifting by $c$. From the point of view of Ahlfors theory of
covering surfaces, the asymptotic behavior of $w(z)$ is
characterized with the help of the covering area under image of $w$ from
a disk, and the change of a disk by a translating can be handled. In
view of this observation, we can establish
$$T(r, w_c)\leq
T(r+|c|,w)+\frac{(2+|c|)\log (1+|c|)}{\log(r+|c|)}T(r+|c|,w)+O(1),$$
here and henceforth, $w_c$ denotes $w(z+c)$. The details will be
provided in Section~2 below. Therefore, by noting that
$w=(w_c)_{-c}$,\ (\ref{equ1.2+}) is naturally transferred into the
discussion of possibility of the equivalent relation
\begin{equation}\label{equ1.8+}T(r+h,w)\sim T(r,w),\end{equation} as $r\to\infty$ on at least a sequence
of positive numbers tending to $+\infty$ for any positive number
$h$. However, the number $h$ can be chosen as $1$ without loss of
generality. Following the discussion of invariance of Nevanlinna
deficiency under a change of the origin given by Goldberg and
Ostrovskii (\cite{GoldbergeOstrovskii}), we can establish the
following result.

\begin{thm}\label{thm1.2} Let $w$ be a transcendental meromorphic
function with
\begin{equation}\label{equ1.7}\liminf\limits_{r\to\infty}\frac{\log
T(r,w)}{r}=0.\end{equation} Then (\ref{equ1.2+}) and (\ref{equ1.8+})
for $T(r,w)$ and $N(r,w)$ hold as $r\not\in E\to\infty,$ where $E$
is a subset of $[1,+\infty)$ with the zero lower density. If
$\liminf$ in (\ref{equ1.7}) is replaced by $\limsup$, then $E$ can
be taken as a set of zero upper density.
\end{thm}

An upper density of a set $E$ in $[1,+\infty)$ is defined as
$$\overline{\rm dens}E=\limsup\limits_{r\to\infty}\frac{1}{r}\int_{E[1,r]}{\rm d}t,$$
where $E[1,r]=E\cap[1,r]$; a lower density $\underline{\rm dens}E$
is the above equality with $\liminf$ replacing $\limsup$.

In fact, (\ref{equ1.7}) in Theorem \ref{thm1.2} is true for all
non-decreasing functions $T(r)$ that are convex with respect to
$\log r$. Above $N(r,w)$ denotes the integrated counting function of $w$.
If $N(r,w)\not=0$, that is to say, $w$ has at least a pole, then
$N(r,w)$ is a non-decreasing convex function with respect to $\log
r$. And $N(r,w)$ satisfies (\ref{equ1.7}) if (\ref{equ1.7}) holds
for $T(r,w)$. Therefore, in this case, Theorem \ref{thm1.2} applies
to $N(r,w)$. By noting that $N(r,w)\leq T(r,w)$, we have
$$T(r,w_c)=(1+\varepsilon(r))T(r,w),\
N(r,w_c)=(1+\varepsilon'(r))N(r,w),$$ where $\varepsilon(r)$(resp.\
$\varepsilon'(r))\to 0 (r\not\in E\to\infty)$ and it is defined in
the proof of Theorem~\ref{thm1.2} and expressed in terms of $T(r,w)$
(resp.\ $N(r,w)$). Then we have
$$m(r,w_c)=(1+\varepsilon(r))m(r,w)+(\varepsilon(r)-\varepsilon'(r))N(r,w).$$
This produces that
\begin{equation}\label{equ1.10}m(r,w_c)-m(r,w)=o(T(r,w)).\end{equation}

From the Clunie type theorems obtained in the literature, we know
that $m(r,w)$ is often small with respect to $T(r,w)$, roughly
speaking, it may be lost sight of. When this happens, $w$ has high
pole density in the sense of Nevanlinna theory, which is a starting
point of making a further classification of some difference
equations by singularity confinement type arguments, but we cannot
confirm $m(r,w_c)\sim m(r,w) (r\not\in E\to\infty)$, basically. In
some cases (see the discussion in Section 4), the equivalent
relation $m(r,w_c)\sim m(r,w) (r\not\in E\to\infty)$ is often
useful.

However, usually the most useful estimate is concerned with the
proximity function of the ratio of $w_c$ and $w$, known as the lemma
of the logarithmic differences, which occupies the same important
position as the lemma of logarithmic derivative does in Nevanlinna
theory. For example, in establishing the Clunie type theorems  and
the Nevanlinna second main theorem for differences
(\cite{RodRistoTohge} and \cite{KLT}), it is a crucial part. Such
lemmas for meromorphic functions of finite order were established
independently by Halburd and the second author in \cite{RodRisto}
and Chiang and Feng in~\cite{ChiangFeng}.

We cannot obtain a valuable estimate of
$m\left(r,\frac{w_c}{w}\right)$ from (\ref{equ1.10}), although we
know that
$$m(r,w_c)\leq m(r,w)+m\left(r,\frac{w_c}{w}\right)$$
and
$$m(r,w)\leq m(r,w_c)+m\left(r,\frac{w}{w_c}\right).$$
Indeed, conversely, an estimate of $m\left(r,\frac{w_c}{w}\right)$
will produce an estimate of $|m(r,w_c)-m(r,w)|$.

In \cite{RodRistoTohge}, Halburd, the second author and Tohge established a version of the lemma of the logarithmic differences
covering the result of finite order and dealing with the case of
infinite order.

\medskip

\noindent {\bf Theorem A.}\ {\sl Let $w(\not\equiv 0)$ be a
meromorphic function with the hyper-order $\zeta=\zeta(w)<1$. Then
for each $\varepsilon>0$, we have
\begin{equation}\label{equ1.10+}
m\left(r,\frac{w_c}{w}\right)=o\left(\frac{T(r,w)}{r^{1-\zeta-\varepsilon}}\right), \
r\not\in E.
\end{equation}
If the order $\rho(w)<+\infty$, then
$$m\left(r,\frac{w_c}{w}\right)\leq\frac{(\log r)^{3+\varepsilon}}{r}T(r,w),\
r\not\in E.$$ Here $E\subset [1,+\infty)$ is a set of finite
logarithmic measure, i.e., $\int_E\frac{{\rm d}t}{t}<\infty.$}

\medskip

The hyper-order $\zeta(w)$ of $w$ is defined by
$$\zeta(w):=\limsup_{r\to\infty}\frac{\log\log T(r,w)}{\log r}$$
and the lower hyper-order by replacing $\limsup$ with $\liminf$.
Obviously, (\ref{equ1.10+}) is invalid when $\zeta=1$. In this
paper, we consider the case when the growth does not exceed the
lower hyper-order $1$ and minimal lower hyper-type, i.e.,
(\ref{equ1.7}). We can establish the following result, which is a
special case of Theorem \ref{thm3.2} below.

\begin{thm}\label{thm1.3+}\ Let $w(\not\equiv 0)$ be a meromorphic function. Assume that (\ref{equ1.7}) holds.
Then for a
complex number $c$ and a real number $\delta$ with $0<\delta<1/2$,
we have
\begin{equation}\label{equ1.11+-}m\left(r,\frac{w_c}{w}\right)\leq 436e(1+|c|)\left(\min_{1\leq t\leq r}\frac{\log
T(t,w)}{t}\right)^{\delta}T(r,w),\ r\not\in E_\delta,\end{equation}
where $E_\delta$ is a subset of $[1,+\infty)$ with $\underline{\rm
dens}E_\delta=0$ and independent of $c$. If (\ref{equ1.7}) holds for
$\limsup$, then $\overline{\rm dens}E_\delta=0$.
\end{thm}

Clearly, (\ref{equ1.7}) implies that
$$\min_{1\leq t\leq r}\frac{\log
T(t,w)}{t}\to 0 \ (r\to\infty).$$ Therefore, from (\ref{equ1.11+-}) it
follows that for any $c$ with $0\leq |c|\leq
o\left(\max\limits_{1\leq t\leq r}\frac{t}{\log
T(t,w)}\right)^\delta$, we have
$$m\left(r,\frac{w_c}{w}\right)=o(T(r,w)),\ r\not\in E_\delta\to\infty.$$
Obviously, we can require $o\left(\max\limits_{1\leq t\leq
r}\frac{t}{\log T(t,w)}\right)^\delta\to\infty \ (r\to\infty).$

The exception set $E_\delta$ has zero lower density, however, the
situation we often meet, for example, in the second main theorem of
Nevanlinna, is that the exception set is of finite logarithmic
measure. We will pay a price on the growth of meromorphic functions if
we ask the exception set of finite logarithmic measure.

\begin{thm}\label{thm1.3}\ Let $w(\not\equiv 0)$ be a meromorphic function. Then for a
complex number $c$, a real number $\delta$ with $0<\delta<1/2$ and
$\varepsilon>0$, we have
\begin{equation*}\label{equ1.11+}m\left(r,\frac{w_c}{w}\right)\leq 436e(1+|c|)
\left(\frac{(\log\log T(r,w))^{1+\varepsilon}\log
T(r,w)}{r}\right)^{\delta}T(r,w),\ r\not\in E,\end{equation*} where
$E$ is a subset of $[1,+\infty)$ with finite logarithmic measure.
\end{thm}

Let us make a remark. If
\begin{equation}\label{equ1.10++}\limsup\limits_{r\to\infty}\frac{(\log r)^{1+\varepsilon}\log
T(r,w)}{r}=0,\end{equation} then
$$\frac{(\log\log T(r,w))^{1+\varepsilon}\log T(r,w)}{r}\to 0
 \ (r\to\infty).$$ The $w$ satisfying (\ref{equ1.10++}) may have the
hyper-order $\zeta(w)=1$. The following is a consequence of Theorem
\ref{thm1.3}.

\begin{thm}\label{thm1.4}\ Let $w(\not\equiv 0)$ be a meromorphic function. Assume that (\ref{equ1.10++}) holds for some
$\varepsilon>0$. Then for $0<\delta<1/2$ and for any complex number
$c$ with
$$0\leq |c|\leq o\left(\frac{r}{(\log
r)^{1+\varepsilon}\log T(r,w)}\right)^\delta,$$ we have
\begin{equation}\label{equ1.13}m\left(r,\frac{w_c}{w}\right)=o(T(r,w)),\ r\not\in E.\end{equation} Here
$E$ is a subset of $[1,+\infty)$ with finite logarithmic measure.
\end{thm}

We can take $o\left(\frac{r}{(\log r)^{1+\varepsilon}\log
T(r,w)}\right)\to\infty (r\to\infty)$ under (\ref{equ1.10++}). Then
$c$ is allowed to tend to $\infty$ as long as $r$ goes to $\infty$.
If $c$ is a fixed complex number, then we still have (\ref{equ1.13})
even if (\ref{equ1.10++}) is replaced by
\begin{equation*}\lim_{\varepsilon\to 0^+}\limsup\limits_{r\to\infty}\frac{(\log r)^{1+\varepsilon}\log
T(r,w)}{r}=0.\end{equation*}
In fact, under this condition, set for $\varepsilon>0$
$$S(\varepsilon)=\limsup_{r\to\infty}\frac{(\log\log T(r,w))^{1+\varepsilon}\log
T(r,w)}{r}.$$ Then
$\lim\limits_{\varepsilon\to0^+}S(\varepsilon)=0$. Assume without
any loss of generality that $S(\varepsilon)>0, \forall\
\varepsilon>0$. There exists an $r_\varepsilon>0$ such that for
$r\geq r_\varepsilon$, we have
$$\frac{(\log\log T(r,w))^{1+\varepsilon}\log
T(r,w)}{r}<2S(\varepsilon).$$ In view of Theorem \ref{thm1.3}, there
exists a subset $E_\varepsilon$ of $[1,+\infty)$ with finite
logarithmic measure such that $$m\left(r,\frac{w_c}{w}\right)\leq
436e(1+|c|)(2S(\varepsilon))^{\delta}T(r,w),\ r\not\in E_\varepsilon
,\ r\geq r_\varepsilon.$$ Take a sequence $\{\varepsilon_n\}$ with
$0<\varepsilon_{n+1}<\varepsilon_n\to 0 \ (n\to\infty)$ and a sequence
$\{r_n\}$ of positive numbers with $r_n\geq r_{\varepsilon_n}$ and
$r_{n+1}\geq r_n\to\infty \ (n\to\infty)$ such that
$$\int_{F_n}\frac{{\rm d}t}{t}<\frac{1}{2^n},\
F_n=E_{\varepsilon_n}\cap[r_n,r_{n+1}).$$ Set
$E=[1,r_1)\cup(\cup_{n=1}^\infty F_n)$. Then $E$ has  finite
logarithmic measure. Then for $r\not\in E$ with $r_n\leq r<r_{n+1}$
we have
$$m\left(r,\frac{w_c}{w}\right)\leq
436e(1+|c|)(2S(\varepsilon_n))^{\delta}T(r,w)=o(T(r,w)),$$ by noting
that $S(\varepsilon_n)\to 0 \ (r\to\infty)$.

\

 In this paper, we will consider the differences of
$\delta$-subharmonic functions. Thus the results we obtain will
apply to the meromorphic functions, the holomorphic curves and the
algebroid functions. That is to say, the results stated above are
special cases of the results we obtain in the sequel.
Therefore, the second main theorems of holomorphic curves for the
difference operator (\cite{RodRistoTohge} and \cite{KLT}) will be
available for the growth not exceeding the (lower) hyper-order one
and minimal (lower) hyper-type. And we can also establish the
difference analogue of the second main theorem for algebroid
functions with the hyper-order at most one and having minimal
hyper-type.

\section{shift invariance of $N_s$ and $T_s$}

A meromorphic function can be considered as a holomorphic curve of
dimension one and an algebroid function can be also expressed by the
coefficients of its corresponding algebraic equation as a
holomorphic curve. However, we can observe holomorphic curves from
the potential point of view. Therefore, this leads us to consider
the difference of subharmonic functions.

Let $u$ be a $\delta$-subharmonic function on $\mathbb{C}$. Let $D$
be a domain on $\mathbb{C}$ surrounded by finitely many piecewise
analytic curves. Then for any point $z\in D$, we have
\begin{equation}\label{equ2.1}\ u(z)=\frac{1}{2\pi}\int_{\partial
D}u(\zeta)\frac{\partial G_D(z,\zeta)}{\partial {\bf n}}{\rm d}s-
\frac{1}{2\pi}\int_{D}G_D(z,w)\Delta u(w),
\end{equation}
where $G_D(z,w)$ is the Green function for $D$ with singularity at
$z$, ${\bf n}$ is the inner normal of $\partial D$ with respect to
$D$ and $\Delta$ is the Laplacian. (Please see \cite{Rainsford} for
the basic theory of subharmonic functions). Define
$$m_s(D,z,u):=\frac{1}{2\pi}\int_{\partial
D}u^+(\zeta)\frac{\partial G_D(z,\zeta)}{\partial {\bf n}}{\rm
d}s,$$
$$N_s(D,z,u):=
\frac{1}{2\pi}\int_{D}G_D(z,w)(\Delta u)^-(w),$$ where
$u^+=\max\{0,u\}$ and $(\Delta u)^-$ is the negative variation of
$\Delta u$, and
$$T_s(D,z,u):=m_s(D,z,u)+N_s(D,z,u).$$
$T_s$ is known as the Nevanlinna characteristic, where the
subscript $s$ stands for ``subharmonic'' function, which is used to distinguish
it from that of a characteristic of a meromorphic function. That is to say, the
functions $m_s, N_s$ and $T_s$ are defined for $\delta$-subharmonic
functions. The Poisson-Jensen formula (\ref{equ2.1}) implies that
$$u(z)=T_s(D,z,u)-T_s(D,z,-u).$$

Now let $D$ be a disk $B(0,r)=\{z:|z|<r\}$. We write $m_s(r,u),
N_s(r,u)$ and $T_s(r,u)$ for $m_s(D,0,u), N_s(D,0,u)$ and
$T_s(D,0,u)$, respectively. As we know, $N_s(r,u)$ is non-decreasing
and logarithmically convex in $r$, i.e., convex in $\log r$ and has
the form
$$N_s(r,u)=\int_0^r\frac{n_s(t,u)-n_s(0,u)}{t}{\rm d}t+n_s(0,u)\log r,$$
where $n_s(t,u)=\frac{1}{2\pi}(\Delta u)^-(B(0,t))$.

Assume that $u$ is subharmonic. Then $(\Delta u)^-=0$ and so
applying (\ref{equ2.1}) to the subharmonic function $u^+$ yields
$$T_s(r,u)=m_s(r,u)=N_s(r,-u^+)+u^+(0).$$
By noting that $-u^+$ is $\delta$-subharmonic, it follows that
$T_s(r,u)$ is non-decreasing and logarithmically convex in $r$. For
a subharmonic function $u$, we often use the Cartan characteristic
which is defined by
$$\mathcal{T}_s(r,u):=\frac{1}{2\pi}\int_0^{2\pi}u(re^{i\theta}){\rm
d}\theta=N_s(r,-u)+u(0).$$ Hence, $\mathcal{T}_s(r,u)$ is
non-decreasing and logarithmically convex in $r$. It is enough to
focus on the discussion of the shift invariance of $N_s$ of $u$ by
$c$ for our purposes in this section.

Let $\mathcal{P}^n(\mathbb{C})$ be the $n$-dimensional complex
projective space, that is,
$\mathcal{P}^n(\mathbb{C})=\mathbb{C}^{n+1}\setminus\{0\}/\sim,$
where $\sim$ is the equivalence relation defined so that
$(a_0,a_1,...,a_n)\sim (b_0,b_1,...,b_n)$ if and only if
$(a_0,a_1,...,a_n)=\lambda(b_0,b_1,...,b_n)$ for some
$\lambda\in\mathbb{C}\setminus\{0\}.$ We write the equivalence class
of $(a_0,a_1,...,a_n)$ as $[a_0:a_1:...:a_n]$.

A map $f:\mathbb{C}\rightarrow \mathcal{P}^n(\mathbb{C})$ is called
a holomorphic curve, if we can write $f=[f_0:f_1:...:f_n]$ where
every $f_j$ is an entire function, at least one of $f_j$ is
non-constant and they have no common zeros on $\mathbb{C}$, and
${\bf{f}}=(f_0,f_1,...,f_n)$ is called a reduced representation of
$f$.

Let $f$ be a holomorphic curve on the complex plane $\mathbb{C}$
with a reduced representation ${\bf f}=(f_0,f_1,...,f_n)$. Set
$$v_f(z)=\bigvee_{j=0}^n\log |f_j(z)|.$$
Then $v_f$ is subharmonic on $\mathbb{C}$. The Cartan characteristic
of $f$ is defined by
$$\mathcal{T}(r,f)=\mathcal{T}_s(r,v_f)- v_f(0)=\frac{1}{2\pi}\int_0^{2\pi}v_f(re^{i\theta}){\rm d}\theta- v_f(0).$$
Then $\mathcal{T}(r,f)$ is a positive logarithmically convex
increasing real-valued function in~$r$. If at least one of $f_j
(j=0,1,2,...,n)$ is transcendental, then $\mathcal{T}(r,f)/\log
r\to\infty \ (r\to\infty)$. We can write the characteristic of an
algebroid function as a special holomorphic curve, by this way. We
can consider a meromorphic function $f$ as a holomorphic curve and
have its characteristic $\mathcal{T}(r,f)$, which is known as the
Ahlfors-Shimizu characteristic. The Nevanlinna characteristic
$T(r,f)$ of a meromorphic function $f$ is
$$T(r,f)=T_s(r,\log|f|),$$
by noting that $\log|f|$ is a $\delta$-subharmonic function. The
following inequality is well known (see, for example, \cite{Hayman}
and \cite{Zheng})
$$|T(r,f)-\mathcal{T}(r,f)|\leq\log 2+\log^+|f(0)|.$$

In a word, all of the following results are valid with
$\mathcal{T}_s(r,u)$ for subharmonic functions $u$, with
$\mathcal{T}(r,f)$ for holomorphic curves $f$ including the
algebroid functions, and with $T(r,f)$ for meromorphic functions
$f$.

\begin{thm}\label{thm2.1}\ Let $u$ be a $\delta$-subharmonic function on
$\mathbb{C}$. Then we have
$$N_s(r,u_c)\sim N_s(r,u),\ \text{as}\ r\not\in E\to\infty,$$
where $E$ is a subset of $[1,+\infty)$ such that
\begin{itemize}
\item[(1)] if
\begin{equation}\label{equ2.2++}\liminf\limits_{r\to\infty}\frac{\log
N_s(r,u)}{r}=0,\end{equation} then $\underline{\rm dens}E=0$;
\item[(2)] if (\ref{equ2.2++}) holds for $\limsup$ instead of $\liminf$, then $\overline{\rm
dens}E=0$;
\item[(3)] if
\begin{equation}\label{equ2.2++-}\limsup\limits_{r\to\infty}\frac{(\log r)^{1+\varepsilon}\log
N_s(r,u)}{r}=0,\end{equation} then $E$ has finite logarithmic
measure.
\end{itemize}
\end{thm}

Let us point out that Theorem~\ref{thm2.1} holds for $T_s$, $\mathcal{T}_s$, $T$ and $\mathcal{T}$ instead of $N_s$. Theorem
\ref{thm1.2} follows from Theorem \ref{thm2.1}.

 In order to prove Theorem
\ref{thm2.1}, we need a basic lemma some of whose idea essentially
comes from \cite{GoldbergeOstrovskii}.

\begin{lem}\label{lem2.1} Let $T(r)$ be a non-decreasing positive
function in $[1,+\infty)$ and logarithmically convex with
$T(r)\to+\infty \ (r\to+\infty)$. Assume that
\begin{equation}\label{equ2.2+}\liminf\limits_{r\to\infty}\frac{\log
T(r)}{r}=0.\end{equation} Set
$$\phi(r)=\max\limits_{1\leq t\leq r}\left\{\frac{t}{\max\{1,\log T(t)\}}\right\}.$$
Then given a real number $\delta$ with $0<\delta<\frac{1}{2}$, we
have
$$T(r+\phi^\delta(r))\leq T(r)+4\phi(r)^{\delta-\frac{1}{2}}T(r),\ r\not\in E_\delta,$$
where $E_\delta$ is a subset of $[1,+\infty)$ with the zero lower
density. And $E_\delta$ has the zero density if (\ref{equ2.2+})
holds for $\limsup$.
\end{lem}

\begin{proof}\
Since $T(r)$ is logarithmically convex, it follows from the basic
inequality of a convex function that for any $t, r\in [1,+\infty)$,
we have
$$T(r)\geq T(t)+\frac{{\rm d}T(t)}{{\rm d}\log t}(\log r-\log t)$$
so that
$$T(t)\leq T(r)+\frac{{\rm d}T(t)}{{\rm d}t}t\log \frac{t}{r}.$$
Take $t=r+\phi^{\delta}(r)$. By noting that $\phi(r)\leq r$, we have
\begin{eqnarray}\label{equ2.4-}T(r+\phi^{\delta}(r))&\leq &
T(r)+\frac{T'(r+\phi^{\delta}(r))}{T(r+\phi^{\delta}(r))}(r+\phi^{\delta}(r))
\log\left(1+\frac{\phi^{\delta}(r)}{r}\right)T(r+\phi^{\delta}(r))\nonumber\\
&\leq&
T(r)+\frac{T'(r+\phi^{\delta}(r))}{T(r+\phi^{\delta}(r))}2\phi^{\delta}(r)T(r+\phi^{\delta}(r)).\end{eqnarray}
Defining
$$\hat{\tau}(r)=\sqrt{\frac{\log^+ T(r)}{r}},\ r\in
[1,+\infty),$$ it follows that $\hat{\tau}$ is continuous in $[1,+\infty)$. From
(\ref{equ2.2+}) we can find a sequence of positive numbers $\{r_n\}$
such that $r_1=1$, $r_n<r_{n+1}\to\infty \  (n\to\infty)$ and
$\hat{\tau}(r_n)=\min\limits_{1\leq t\leq r_n}\hat{\tau}(t)$. It is
obvious that $0<\hat{\tau}(r_{n+1})\leq \hat{\tau}(r_n)\to 0 \ 
(n\to\infty).$ Define $\tau(r)=\hat{\tau}(r_n)$, $r_{n-1}< r\leq
r_{n}$. It is easy to see that $\tau(r)$ is non-increasing in
$[1,+\infty)$, $\tau(r)=\phi(r)^{-1/2}$ for $r=r_n$ and
$\tau(r)\leq\phi(r)^{-1/2}\to 0 \ (r\to\infty)$, as $\phi(r)\to\infty \ 
(r\to\infty)$ from (\ref{equ2.2+}).

Set
$$F=\left\{r\in [1,+\infty):\ \frac{T'(r)}{T(r)}\geq
\tau(r)\right\}.$$ There is an $r_0>1$ such that $T(r_0)\geq 1$. Then
\begin{eqnarray*} \log T(r)&\geq &\int_{r_0}^{r}(\log
T(t))'_t{\rm d}t\\
&=&\int_{r_0}^{r}\frac{T'(t)}{T(t)}{\rm d}t\\
&\geq &\int_{F[r_0,r]}\frac{T'(t)}{T(t)}{\rm d}t\\
&\geq &\int_{F[r_0,r]}\tau(t){\rm d}t\\
&\geq&\tau(r)\int_{F[r_0,r]}{\rm d}t,
\end{eqnarray*}
where $F[r_0,r]=F\cap [r_0,r]$ so that, when $r=r_n$, we have
\begin{eqnarray*}\frac{1}{r}\int_{F[1,r]}{\rm
d}t&=&\frac{1}{r}\int_{F[1,r_0]}{\rm
d}t+\frac{1}{r}\int_{F[r_0,r]}{\rm d}t\\
&\leq&\frac{1}{r}\int_1^{r_0}{\rm d}t+
\frac{\hat{\tau}^2(r)}{\tau(r)}\\
&=&\frac{r_0-1}{r}+\tau(r)\to 0 \  (r\to\infty).\end{eqnarray*} This
yields that the lower density of $F$ satisfies $\underline{{\rm
dens}}(F)=0$. Set $E_\delta=\{r:r+\phi^\delta(r)\in F\}$. Since
$r+\phi^\delta(r)\sim r \ (r\to\infty)$, we have $\underline{{\rm
dens}}(E_\delta)=0$. There is an $R_0>1$ such that
$2\phi^\delta(r)\tau(r)\leq 2\phi^{\delta-1/2}(r)<1/2$ for $r>R_0$.
Then for $\forall\ r\not\in E_\delta\cup[1,R_0]$, we have
\begin{eqnarray*}T(r+\phi^\delta(r))&\leq&
T(r)+2\phi^\delta(r)\tau(r+\phi^\delta(r))T(r+\phi^\delta(r))\\
&\leq&
T(r)+2\phi^\delta(r)\tau(r)T(r+\phi^\delta(r))\\
&<&T(r)+\frac{1}{2}T(r+\phi^\delta(r)),\end{eqnarray*} which reduces
to $T(r+\phi^\delta(r))<2T(r)$. Thus for $r\not\in
E_\delta\cup[1,R_0]$,
$$T(r+\phi^\delta(r))\leq T(r)+4\phi^\delta(r)\tau(r)T(r)\leq T(r)+4\phi(r)^{\delta-1/2}T(r).$$
\end{proof}

Let us make a remark on Lemma \ref{lem2.1}. As we know, there are a few lemmas on the growth of real functions that were established in 
Nevanlinna theory, so we wonder if from them we could obtain Lemma
\ref{lem2.1}. After checking this, we have not found the possibility. Set
$$E=\{r\geq 1:\ T(r+1)\geq CT(r)\},\ C>1.$$
In view of Lemma 3.3.1 in \cite{CherryYe}, we have
\begin{eqnarray*}\int_{E[e,r]}\frac{{\rm d}t}{t}&\leq
&\frac{1}{T^{-1}(e)}+\frac{1}{\log C}\int_e^{T(r)}\frac{{\rm
d}t}{tT^{-1}(t)}\\
&=&\frac{1}{T^{-1}(e)}+\frac{1}{\log
C}\int_{T^{-1}(e)}^{r}\frac{T'(t)}{T(t)t}{\rm
d}t\\
&=&\frac{1}{T^{-1}(e)}+\frac{1}{\log C}\left(\frac{\log
T(r)}{r}-\frac{1}{T^{-1}(e)}+\int_{T^{-1}(e)}^{r}\frac{\log
T(t)}{t^2}{\rm d}t\right).\end{eqnarray*} Therefore, if
(\ref{equ2.2+}) holds for $\limsup$, then $E$ has the upper
logarithmic density $\overline{\log dens}(E)=0$. As we know that
$\overline{\log dens}(E)\leq \overline{dens}(E)$, from Lemma 3.3.1
in~\cite{CherryYe} we cannot have $\overline{dens}(E)=0$. And from
$T(r+1)<CT(r)$ with $C>1$ for $r\not\in E$ we cannot directly obtain
$T(r+1)\sim T(r) \ (r\to\infty).$

If the growth of $T(r)$ is a little slower, we can obtain that the
exceptional set $E_\delta$ has finite logarithmic measure, which
appears often in the second main theorems of Nevanlinna. The
following is Lemma 10.1 of Edrei and Fuchs \cite{EdreiFuchs}.

\begin{lem}\label{lem1.1}\ Let $\psi$ and $\varphi$ be two positive
functions on $[r_0,+\infty)$ with $r_0>0$. Assume that $\psi$ is
non-decreasing while $\varphi$ is non-increasing and that $\psi(r_1)>r_0+1$ for some
$r_1>r_0$. Set
$$E=\{r\geq r_1:\ \psi(r+\varphi(\psi(r)))\geq\psi(r)+1\}.$$
Then we have
$$\int_{E[a,A]}{\rm
d}t\leq\int_{\psi(a)-1}^{\psi(A)}\varphi(t){\rm d}t,$$ provided that
$r_1\leq a<A<+\infty$, where $E[a,A]=E\cap [a,A]$.
\end{lem}

For a non-decreasing positive function $T(r)$ on $[1,+\infty)$ with
$T(r)>e$, define
$$\psi(r)=\log T(e^r),\ \varphi(r)=\frac{1}{r(\log
r)^{1+\varepsilon}}.$$ Then
\begin{eqnarray*}\ \psi(\log r+\varphi(\psi(\log r)))&=&\log T\left(re^{1/(\log
T(r)(\log\log T(r))^{1+\varepsilon})}\right)\\
&\leq &\psi(\log
r)+1\\
&=&\log eT(r), \quad \log r\not\in F_\varepsilon,
\end{eqnarray*}
where $F_\varepsilon$ is a subset of $[0,+\infty)$ such that
$$\int_{F_\varepsilon[1,r]}{\rm
d}t\leq\int_{1}^{\infty}\varphi(t){\rm d}t<+\infty,$$ that is to
say, $F_\varepsilon$ has finite linear measure. Set
$E_\varepsilon=\{r>e:\ \log r\in F_\varepsilon\}$. Then
$E_\varepsilon$ has finite logarithmic measure and
\begin{equation}\label{equ2.7}T\left(re^{1/(\log T(r)(\log\log
T(r))^{1+\varepsilon})}\right)\leq eT(r),\ r\not\in
E_\varepsilon.\end{equation} Here and henceforth, define
$$\phi_{T,\varepsilon}(r)=\frac{r}{(\log\log T(r))^{1+\varepsilon}\log T(r)}$$
for a non-decreasing positive function $T(r)$ with $T(r)>e$ and
$\varepsilon>0$. Without confusion in the context, we will write
$\phi_\varepsilon$ for $\phi_{T,\varepsilon}$. Then (\ref{equ2.7})
produces
\begin{equation}\label{equ3.7-}T(r+\phi_\varepsilon(r))\leq eT(r),\ r\not\in
E_\varepsilon.\end{equation}

As an application of (\ref{equ3.7-}), we establish the following,
which is a supplement of Lemma \ref{lem2.1}. We emphasize that these
two lemmas apply to the discussion of infinite shifting.

\begin{lem}\label{lem2.2}\ Let $A$ be a non-decreasing positive function on
$[1,+\infty)$ and
$$T(r)=\int_1^r\frac{A(t)}{t}{\rm d}t,\ r\geq 1.$$ Set
$\phi_\varepsilon(r)=\phi_{T,\varepsilon}(r)$ for a given
$\varepsilon>0$. Then for all $1\leq r<R\leq
r+\phi_\varepsilon(r)^\delta$, $0<\delta< 1$, we have
\begin{equation}\label{equ2.8}
T(R)\leq
T(r)+4e\left(\frac{1}{\phi_\varepsilon(r)}\right)^{1-\delta}T(r),\
r\not\in E_\varepsilon,\end{equation} where
$\int_{E_\varepsilon}\frac{{\rm d}t}{t}<+\infty$ if
$\phi_\varepsilon(r)>2^{1/(1-\delta)}$ and $\phi_\varepsilon(r)<r$.
If $R=r+h$ for a fixed positive real number $h$, then $\delta=0$ can
be chosen in (\ref{equ2.8}) and in addition, if $T(r)$ is of finite
order $\rho$, then
$$T(r+h)\leq T(r)+4hK\frac{T(r)}{r},\ r\not\in E,$$ where $E$ has
the logarithmic density less than $\rho\frac{\log 2}{\log K}$.
\end{lem}

\begin{proof}\ From the definition of $T$, it follows that
\begin{eqnarray*}\ T(R)&=&\int_1^R\frac{A(t)}{t}{\rm
d}t=T(r)+\int_r^R\frac{A(t)}{t}{\rm d}t\\
&\leq &T(r)+A(R)\log\frac{R}{r}\\
&\leq & T(r)+\frac{R-r}{r}\frac{1}{\log\alpha}\int_R^{\alpha
R}\frac{A(t)}{t}{\rm d}t\\
&\leq &T(r)+\frac{\alpha}{\alpha-1}\frac{R-r}{r}T(\alpha R),
\end{eqnarray*}
with $\alpha>1$ satisfying $\alpha R=r+\phi(r).$ Here and
henceforth, simply we write $\phi(r)$ for $\phi_\varepsilon(r)$. In
terms of (\ref{equ3.7-}), we have
$$T(\alpha R)=T(r+\phi(r))\leq eT(r), \ r\not\in
E_\varepsilon,$$ where $E_\varepsilon$ has  finite logarithmic
measure. Thus
\begin{eqnarray*}
\frac{\alpha}{\alpha-1}\frac{R-r}{r}&=&\frac{r+\phi(r)}{r+\phi(r)-R}\frac{R-r}{r}\\
&\leq
&\frac{r+\phi(r)}{\phi(r)-\phi(r)^\delta}\frac{\phi(r)^\delta}{r}\\
&=&\frac{r+\phi(r)}{r(\phi(r)^{1-\delta}-1)}\\
&\leq &4\left(\frac{1}{\phi(r)}\right)^{1-\delta},
\end{eqnarray*}
if $\phi^{1-\delta}(r)>2$ and $\phi(r)<r$. Combining the above
inequalities yield (\ref{equ2.8}). Let us take $R=r+h$, specially.
Then we have
\begin{eqnarray*}\frac{\alpha}{\alpha-1}\frac{R-r}{r}T(\alpha
R)&\leq&\frac{r+\phi(r)}{\phi(r)-h}\frac{h}{r}eT(r)\\
&\leq&\frac{2eh}{\phi(r)-h}T(r)\\
&\leq &\frac{4eh}{\phi(r)}T(r),
\end{eqnarray*}
if $\phi(r)>2h$ and $r\not\in E_\varepsilon$. If we take $\alpha$
such that $\alpha R=2r$ and $R=r+h$, then we have
$$\frac{\alpha}{\alpha-1}\frac{R-r}{r}T(\alpha
R)=\frac{2h}{r-h}T(2r)<4h\frac{T(2r)}{r}.$$ Therefore, if the order
$\rho$ of $T(r)$ is finite, then in view of a result in Hayman
\cite{Hayman1}, we have
$$T(2r)\leq KT(r),\ r\not\in W,$$
where $W$ has $\overline{\log{\rm dens}}W\leq\rho\frac{\log 2}{\log
K}.$ This produces the desired result.
\end{proof}

Now we are in position to prove Theorem \ref{thm2.1}.

\

\textit{Proof of Theorem 2.1.}\ Set $\mu=\frac{1}{2\pi}(\Delta u)^-$
and $\mu_c=\frac{1}{2\pi}(\Delta u_c)^-$. Since $B(0,t)+c\subset
B(0,t+|c|)$, we have
$$n_s(t,u_c)=\mu_c(B(0,t))=\mu(B(0,t)+c)\leq \mu(B(0,t+|c|))=n_s(t+|c|,u),$$
by noting that $\mu$ is a Borel measure. From the definition of
$N_s(r,u_c)$ it follows that
\begin{eqnarray*}N_s(r,u_c)&=&\int_1^r\frac{n_s(t,u_c)}{t}{\rm d}t+\int_0^1\frac{n_s(t,u_c)-n_s(0,u_c)}{t}{\rm d}t\nonumber\\
&\leq &\int_1^r\frac{n_s(t+|c|,u)}{t}{\rm d}t+\int_0^1\frac{n_s(t,u_c)-n_s(0,u_c)}{t}{\rm d}t\nonumber\\
&=&\int_{1+|c|}^{r+|c|}\frac{n_s(t,u)}{t-|c|}{\rm d}t+\int_0^1\frac{n_s(t,u_c)-n_s(0,u_c)}{t}{\rm d}t\nonumber\\
&=&N_s(r+|c|,u)+|c|\int_{1+|c|}^{r+|c|}\frac{n_s(t,u)}{t(t-|c|)}{\rm
d}t-n_s(0,u)\log(1+|c|)\end{eqnarray*}
\begin{eqnarray}\label{equ2.3+}
&\ &-\int_0^{1+|c|}\frac{n_s(t,u)-n_s(0,u)}{t}{\rm d}t+\int_0^1\frac{n_s(t,u_c)-n_s(0,u_c)}{t}{\rm d}t\nonumber\\
&=&N_s(r+|c|,u)+\frac{|c|}{r}N_s(r+|c|,u)+|c|\int_{1+|c|}^{r+|c|}\frac{N_s(t,u)}{(t-|c|)^2}{\rm
d}t+O(1),\end{eqnarray} where the equality
$\frac{1}{t-|c|}=\frac{1}{t}+\frac{|c|}{t(t-|c|)}$ was used and the
last equality resulted from the formula of integration by parts.

Since $N_s(r,u)$ is logarithmically convex,
$\frac{N_s(r,u)-N_s(1,u)}{\log r}$ is increasing for $r\geq 1$.
Therefore, for $1+|c|\leq t\leq r+|c|$, we have
$$\frac{N_s(t,u)-N_s(1,u)}{\log t}\leq
\frac{N_s(r+|c|,u)-N_s(1,u)}{\log(r+|c|)}.$$ This implies that
$$\frac{N_s(t,u)}{\log t}\leq\frac{N_s(1,u)}{\log t}+\frac{N_s(r+|c|,u)}{\log (r+|c|)}$$ so
that
\begin{eqnarray*}\int_{1+|c|}^{r+|c|}\frac{N_s(t,u)}{(t-|c|)^2}{\rm
d}t&\leq&\frac{N_s(r+|c|,u)}{\log
(r+|c|)}\int_{1+|c|}^{r+|c|}\frac{\log t}{(t-|c|)^2}{\rm d}t+N_s(1,u)\\
&\leq & \frac{\int_{1}^{r}\frac{\log (t+|c|)}{t^2}{\rm d}t}{\log
(r+|c|)}N_s(r+|c|,u)+N_s(1,u).\end{eqnarray*}

Noting that $$\int_1^\infty\frac{\log (t+|c|)}{t^2}{\rm
d}t=-\frac{\log (t+|c|)}{t}|_1^\infty+\int_1^\infty\frac{{\rm
d}t}{t(t+|c|)}$$$$=\log (1+|c|)+\frac{1}{|c|}\log (1+|c|),$$ from
the above inequalities it follows that
\begin{equation}\label{equ2.6-}
N_s(r,u_c)\leq \left(1+\frac{|c|}{r}+\frac{(1+|c|)\log(1+|c|)}{\log
(r+|c|)}\right)N_s(r+|c|,u)+N_s(1,u).
\end{equation}

From (\ref{equ2.2++}) it follows that $\phi(r)=\max\limits_{1\leq
t\leq r}\left\{\frac{t}{\max\{1,\log N_s(t,u)\}}\right\}\to\infty \
(r\to\infty)$. In terms of Lemma \ref{lem2.1}, we have
$$N_s(r,u_c)\leq \left(1+\varepsilon_1(r)\right)N_s(r,u), \ r\in F_1,$$
for a set $F_1$ of $r$ with upper density $1$ and
$\varepsilon_1(r)\to 0 \ (r\to\infty)$. The same argument implies that
$$N_s(r,u)\leq \left(1+\varepsilon_2(r)\right)N_s(r,u_c), \ r\in F_2,$$
for a set $F_2$ of $r$ with upper density $1$ and
$\varepsilon_2(r)\to 0 (r\to\infty)$. Then
$$N_s(r,u)\sim N_s(r,u_c), \text{as}\ r\in F_1\cap F_2\to\infty.$$
It is obvious that $F=F_1\cap F_2$ has upper density $1$ and
$E=[1,+\infty)\setminus F$ has zero lower density. The result (2)
follows the same way as above.

In view of Lemma \ref{lem2.2}, we can deduce the result (3) along
the same way as above, for from (\ref{equ2.2++-}) it follows that
$\phi_{N_s,\varepsilon}(r)\to\infty \ (r\to\infty).$

The proof of Theorem~\ref{thm2.1} is completed. 

\

In order to obtain a finer estimate than (\ref{equ2.6-}), we should
carefully discuss the estimate of the integral term in
(\ref{equ2.3+}).

If $N_s(r,u)$ has finite order $\lambda$, then it follows that for
$\varepsilon>0$ there exists a constant $K_\varepsilon>0$ such that
$N_s(r,u)<K_\varepsilon(r-|c|)^{\lambda+\varepsilon},\ r\geq 1+|c|.
$ For $\varepsilon>0$ with $\lambda-1+\varepsilon\not=0$, we can
estimate
\begin{eqnarray*}\ \int_{1+|c|}^{r+|c|}\frac{N_s(t,u)}{(t-|c|)^2}{\rm
d}t&\leq&K_\varepsilon\int_{1+|c|}^{r+|c|}\frac{(t-|c|)^{\lambda+\varepsilon}}{(t-|c|)^2}{\rm
d}t\\
&\leq&\frac{K_\varepsilon}{|\lambda-1+\varepsilon|}\left(1+r^{\lambda-1+\varepsilon}\right).
\end{eqnarray*}
This together with (\ref{equ2.3+}) produces
$$N_s(r,u_c)\leq N_s(r+|c|,u)+O(r^{\lambda-1+\varepsilon})+O(1).$$
By noting that
$$N_s(r+|c|,u)-N_s(r,u)=\int_r^{r+|c|}\frac{n_s(t,u)}{t}{\rm d}t$$
$$\leq
n_s(r+|c|,u)\log\left(1+\frac{|c|}{r}\right)=O(r^{\lambda-1+\varepsilon}),$$
we have
\begin{equation}\label{equ2.11+-}N_s(r+|c|,u)=N_s(r,u)+O(r^{\lambda-1+\varepsilon})+O(1).\end{equation} The
same arguments as above yield the above equality (\ref{equ2.11+-})
for $\mathcal{T}_s(r,u)$ of subharmonic functions $u$ with $\lambda$
replaced by the order $\rho$, especially for $\mathcal{T}(r,f)$ of
holomorphic curves $f$ including the algebroid functions and
$T(r,f)$ of meromorphic functions $f$. The asymptotic identity (\ref{equ2.11+-}) was
obtained by Chiang and Feng~\cite{ChiangFeng}. Using the lemma of
logarithmic differences as a bridge, that is,
$$m\left(r,\frac{w_c}{w}\right)=O(r^{\rho-1+\varepsilon}),$$ they
imply that
$$T(r,w_c)=T(r,w)+O(r^{\rho-1+\varepsilon})+O(\log r).$$
Our methods in this paper are different from theirs. It is obvious
that the above equalities make no sense for $r$ with
$T(r,w)=o(r^{\rho-1+\varepsilon})$. For example, if the lower order
of $w$ is less than $\rho-1$, then there exists an unbounded set $F$
of $r$ such that $T(r,w)=o(r^{\rho-1+\varepsilon}),\ r\in F$.

In view of Lemma 1.1.2 in \cite{Zheng}, if
\begin{equation}\label{equ2.7--}\liminf\limits_{r\to\infty}\frac{N_s(dr,u)}{N_s(r,u)}\geq
d\end{equation} for some $d>1$, then
$$\int_{1+|c|}^{r+|c|}\frac{N_s(t,u)}{(t-|c|)^2}{\rm
d}t\leq K\frac{N_s(r+|c|,u)}{r}+O(1).$$ Thus (\ref{equ2.3+}) yields
\begin{equation}\label{equ2.12+-}N_s(r,u_c)\leq
\left(1+\frac{|c|(K+2)}{r}\right)N_s(r+|c|,u).\end{equation}

We also deal with the finite lower order. Assume that $N_s(r,u)$ has
the lower order $\mu<+\infty$ and the order $\lambda>1$. Then in
view of Theorem 1.1.3 in \cite{Zheng}, for any fixed set $E$ of the
finite logarithmic measure and $\max\{1,\mu\}<\rho\leq\lambda$,
there exists a sequence of the Poly\'a peak $\{r_n\}$ of order
$\rho$ outside $E$ such that
$$N_s(r_n+|c|,u)\leq (1+\varepsilon_n)\left(1+\frac{|c|}{r_n}\right)^{\rho}N_s(r_n,u),$$
$$\frac{N_s(t,u)}{t^{\rho-\varepsilon_n}}\leq
K\frac{N_s(r_n,u)}{r_n^{\rho-\varepsilon_n}}$$ for $1\leq t\leq
r_n+|c|$ and $\varepsilon_n\to 0 \ (n\to\infty).$ Then we have for
$r=r_n$ \begin{eqnarray*}\int_1^r\frac{N_s(t+|c|,u)}{t^2}{\rm
d}t&\leq&K\frac{N_s(r_n,u)}{r_n^{\rho-\varepsilon_n}}\int_1^r\frac{(t+|c|)^{\rho-\varepsilon_n}}{t^2}{\rm
d}t\\
&\leq&
K\frac{N_s(r_n,u)}{r_n^{\rho-\varepsilon_n}}(1+|c|)^\rho\int_1^rt^{\rho-2-\varepsilon_n}{\rm
d}t\\
&\leq&
K(1+|c|)^\rho\frac{N_s(r_n,u)}{r_n^{\rho-\varepsilon_n}}\frac{1}{|\rho-1-\varepsilon_n|}
\left[1+r_n^{\rho-1-\varepsilon_n}\right]\\
&\leq&
\frac{2K(1+|c|)^\rho}{|\rho-1-\varepsilon_n|}\frac{N_s(r_n,u)}{r_n},\end{eqnarray*}
for all sufficiently large $n$ by noting that $r_n^{\rho-1-\varepsilon_n}\geq 1$.

For the case of the infinite order, in view of Lemma 1.1.3 in
\cite{Zheng}, there exist these~$\{r_n\}$ such that the above
inequality holds.

Throughout the above discussion, we assume that $c$ is a constant
independent of~$r$. Now we consider the case when $c$ is allowed to
depend on $r$, as Chiang and Luo did in \cite{ChiangLuo}. In view of
Lemma \ref{lem2.1}, we can establish the following

\begin{thm}\label{thm2.1-}\ Let $u$ be a $\delta$-subharmonic function on
$\mathbb{C}$ with (\ref{equ2.2++}). Assume that (\ref{equ2.7--})
holds. Then for all $c$ with $0<|c|<\phi^\delta(r)$ and
$0<\delta<1$, we have
\begin{equation}\label{equ--}N_s(r,u_c)\sim N_s(r,u),\ \text{as}\ r\in
E_\delta\to\infty,\end{equation} where $E_\delta$ is such that
$\underline{\rm dens}E_\delta=0$.
\end{thm}

\begin{proof}\ Under (\ref{equ2.7--}), we have (\ref{equ2.12+-}).
Applying Lemma \ref{lem2.1} for $T=N_s$, we get
$$N_s(r,u_c)\leq
\left(1+\frac{|c|(K+2)}{r}\right)
\left(1+4\left(\frac{1}{\phi(r)}\right)^{\frac{1-\delta}{2}}\right)N_s(r,u)\
(r\not\in E_\delta\to\infty).$$ By noting that $n_s(t-|c|,u)\leq
n(t, u_c)$£¬ we can establish the opposite inequality. Therefore, we
get (\ref{equ--}).
\end{proof}

We guess that (\ref{equ2.7--}) would not be necessary. But we have
to look for another way to get (\ref{equ--}) without
(\ref{equ2.7--}).

\section{Difference of a Logarithm}

Let $u$ be a $\delta$-subharmonic function on $\mathbb{C}$. Let $D$
be a domain on $\mathbb{C}$ surrounded by finitely many piecewise
analytic curves. For any point $z\in D$ and a fixed complex
number $c$ such that $z+c\in D$, we set $u_c(z)=u(z+c)$. Then, in view
of the Poisson-Jensen formula (\ref{equ2.1}), we have
\begin{eqnarray*}\ u_c(z)-u(z)&=&\frac{1}{2\pi}\int_{\partial
D}u(\zeta)\left(\frac{\partial G_D(z+c,\zeta)}{\partial {\bf
n}}-\frac{\partial G_D(z,\zeta)}{\partial {\bf n}}\right){\rm d}s\\
&\ &+N_s(D,z+c,u)-N_s(D,z,u)\\
&\ &+N_s(D,z,-u)-N_s(D,z+c,-u),
\end{eqnarray*}
so that, by noting that $G_D(z,w)>0$ for all $w \in D$ and using the fact that $(\Delta
u)^-$ is a Borel measure, we have
\begin{eqnarray}\label{equ2.2}
(u_c(z)-u(z))^+&\leq &\frac{1}{2\pi}\int_{\partial
D}\left[u(\zeta)\left(\frac{\partial G_D(z+c,\zeta)}{\partial {\bf
n}}-\frac{\partial G_D(z,\zeta)}{\partial {\bf n}}\right)\right]^+{\rm d}s\nonumber\\
&\ &+(N_s(D,z+c,u)-N_s(D,z,u))^+\nonumber\\
&\ &+(N_s(D,z,-u)-N_s(D,z+c,-u))^+\nonumber\\
&\leq &\frac{1}{2\pi}\int_{\partial D}u^+(\zeta)\left(\frac{\partial
G_D(z+c,\zeta)}{\partial {\bf n}}-\frac{\partial
G_D(z,\zeta)}{\partial {\bf n}}\right)^+{\rm
d}s\nonumber\\
&\ &+\frac{1}{2\pi}\int_{\partial
D}(-u)^+(\zeta)\left(\frac{\partial G_D(z,\zeta)}{\partial {\bf
n}}-\frac{\partial G_D(z+c,\zeta)}{\partial {\bf n}}\right)^+{\rm d}s\nonumber\\
\label{equ2.3} &\ &+N_s(D,z+c,u)+N_s(D,z,-u).
\end{eqnarray}

Let $B$ be a domain in $D$ surrounded by finitely many piecewise
analytic curves. We want to estimate $m_s(B,a,u_c-u)$ for a fixed
point $a\in B$. For a fixed point $w\in D$, in terms of the
Poisson-Jensen formula (\ref{equ2.1}), we have
\begin{eqnarray*}\ m_s(B,a,G_D(z,w))&=&G_D(a,w)+\frac{1}{2\pi}\int_B
G_B(a,\zeta)\Delta_\zeta G_D(\zeta,w)\\
&=&G_D(a,w)-\frac{1}{2\pi}\int_B
G_B(a,\zeta)\Delta_\zeta \log|\zeta-w|\\
&=&G_D(a,w)-\int_B
G_B(a,\zeta) \delta_w\\
&=&G_D(a,w)-G_B(a,w)\chi_B(w),
\end{eqnarray*}
and \begin{eqnarray*}\
m_s(B,a,G_D(z+c,w))&=&G_D(a+c,w)-G_{B}(a,w-c)\chi_{B}(w-c)\\
&=&G_D(a+c,w)-G_{B+c}(a+c,w)\chi_{B+c}(w),\end{eqnarray*}
where
$\Delta_\zeta=\Delta$ is the Laplacian with respect to $\zeta$,
which is emphasized by writing $\zeta$ as a subscript, $\delta_w$ is
the Dirac delta function at $w$ and $\chi_B$ is the characteristic
function of $B$. For simplicity, define $N_{D,u}(z)=N_s(D,z,u)$ and
$N_{D,u,c}(z)=N_s(D,z+c,u)$.
Then
\begin{eqnarray}\label{equ2.4}\ m_s(B,a,N_{D,u})&=&\frac{1}{2\pi}\int_D
m_s(B,a,G_{D}(z,w))(\Delta u)^-(w)\nonumber\\
&=&\frac{1}{2\pi}\int_D
(G_{D}(a,w)-G_{B}(a,w)\chi_B(w))(\Delta u)^-(w)\nonumber\\
&=&N_s(D,a,u)-N_s(B,a,u)
\end{eqnarray}
and
\begin{eqnarray}\label{equ2.4+}\ m_s(B,a,N_{D,u,c})&=&\frac{1}{2\pi}\int_D
m_s(B,a,G_{D}(z+c,w))(\Delta u)^-(w)\nonumber\\
&=&\frac{1}{2\pi}\int_D
(G_{D}(a+c,w)-G_{B+c}(a+c,w)\chi_{B+c}(w))(\Delta u)^-(w)\nonumber\\
&=&N_s(D,a+c,u)-N_s(B+c,a+c,u).
\end{eqnarray}

Combining the above inequalities (\ref{equ2.3}), (\ref{equ2.4}) and
(\ref{equ2.4+}) yields the following

\begin{thm}\label{thm2.1+}\ Let $u$ be a $\delta$-subharmonic function on $D$ and
$B$ a domain such that $B\subset D$ and $B+c\subset D$. Then for any
$a\in B$, we have
\begin{eqnarray*}m_s(B,a,u_c-u)&\leq&\frac{1}{2\pi}\int_{\partial D}|u(\zeta)|m_s\left(B,a,\left|\frac{\partial G_D(z+c,\zeta)}{\partial {\bf
n}}-\frac{\partial G_D(z,\zeta)}{\partial {\bf
n}}\right|\right){\rm d}s\\
&\ & +N_s(D,a+c, u)-N_s(B+c,a+c,u)\\
&\ &+N_s(D,a, -u)-N_s(B,a,-u).
\end{eqnarray*}
\end{thm}

Assume that $D$ is simply connected. Let $\phi_a:\ \mathbb{D}\to D$
be the Riemann mapping with $\phi_a(0)=a,\ \phi'_a(0)>0$ and let
$\psi_a$ be the inverse of $\phi_a$. Then
$$\frac{\partial G_D(z,\zeta)}{\partial {\bf
n}}{\rm d}s=-i\frac{\psi'_z(\zeta)}{\psi_z(\zeta)}{\rm d}\zeta,$$ so
that
$$\left(\frac{\partial
G_D(z+c,\zeta)}{\partial {\bf n}}-\frac{\partial
G_D(z,\zeta)}{\partial {\bf n}}\right)^+{\rm
d}s=\left(\frac{\psi'_{z+c}(\zeta)\psi_z(\zeta)}{\psi_{z+c}(\zeta)\psi'_z(\zeta)}-1\right)^+
\frac{\partial G_D(z,\zeta)}{\partial {\bf n}}{\rm d}s,$$ where
$\zeta\in
\partial D.$
Noting that $|b|=b^++(-b)^+$ for a complex number $b$, we have
\begin{equation}\label{equ2.5}m_s\left(B,a,\left|\frac{\partial G_D(z+c,\zeta)}{\partial {\bf
n}}-\frac{\partial G_D(z,\zeta)}{\partial {\bf
n}}\right|\right)\end{equation}
$$=m_s\left(B,a,\frac{\partial G_D(z+c,\zeta)}{\partial
{\bf n}}-\frac{\partial G_D(z,\zeta)}{\partial {\bf
n}}\right)+m_s\left(B,a,\frac{\partial G_D(z,\zeta)}{\partial {\bf
n}}-\frac{\partial G_D(z+c,\zeta)}{\partial {\bf n}}\right).$$
From the definition of $N_s$, we have
\begin{eqnarray*}\ N_s(D,c,u)&=&\frac{1}{2\pi}\int_DG_D(c,w)(\Delta
u)^-(w)\nonumber\\
&=&\frac{1}{2\pi}\int_DG_{D+c}(0,w-c)(\Delta
u_c)^-(w-c)\nonumber\\
&=&\frac{1}{2\pi}\int_{D+c}G_{D+c}(0,w)(\Delta
u_c)^-(w)\nonumber\\
&=&N_s(D+c,0,u_c).\end{eqnarray*} Therefore, in view of the
subordinate principle of the Green function, for simply connected
domains $\hat{D}$ and $\tilde{D}$ such that $0\in\tilde{D}\subset
D+c\subset\hat{D}$, we have \begin{equation}\label{equ3.6+}
N_s(\tilde{D},0,u_c)\leq N_s(D,c,u)\leq
N_s(\hat{D},0,u_c).\end{equation}

These formulas apply to some special domains such as a disk or a
sector. In this paper, we only establish Lemma of the logarithmic
differences on a disk.

We discuss the case when $D=B(0,R)$ and $B=B(0,r)$ with $0<r<R$,
where $B(0,r)$ denotes the disk centered at $0$ with radius $r$.
Then for $\zeta=Re^{i\theta}$ and $z=re^{i\varphi}$, we have
\begin{eqnarray*}&\
&\left(\frac{\partial G_D(z+c,\zeta)}{\partial
{\bf n}}-\frac{\partial G_D(z,\zeta)}{\partial {\bf n}}\right)^+{\rm d}s\\
&=&\left(\frac{R^2-|z+c|^2}{|\zeta-z-c|^2}-\frac{R^2-|z|^2}{|\zeta-z|^2}\right)^+{\rm d}\theta\\
&=&\frac{R^2-|z|^2}{|\zeta-z|^2}\left(\frac{R^2-|z+c|^2}{|\zeta-z-c|^2}\frac{|\zeta-z|^2}{R^2-|z|^2}-1\right)^+{\rm d}\theta\\
 &\leq
&\frac{R^2-|z|^2}{|\zeta-z|^2}\left(\left(1+\frac{|z|^2-|z+c|^2}{R^2-|z|^2}\right)\frac{|\zeta-z|^2}{|\zeta-z-c|^2}-1\right)^+{\rm
d}\theta
\end{eqnarray*}\begin{eqnarray*}
&\leq&\frac{R^2-|z|^2}{|\zeta-z|^2}\left(\left(1+\frac{3|c||z|}{R^2-|z|^2}\right)\left(1+\frac{|c|}{|\zeta-z-c|}\right)^2-1\right){\rm d}\theta\\
&\leq&\left(\left(1+\frac{2|c|}{R-|z|}\right)\left(1+\frac{2}{|\zeta-z|}\right)^2-1\right)\frac{R^2-|z|^2}{|\zeta-z|^2}{\rm
d}\theta,
\end{eqnarray*}
where we assume that $|\zeta-z|\geq \max\{1,2|c|\}$ and $|z|>|c|$.
For simplicity, set $P=\frac{2|c|}{R-|z|}$ and
$Q=\frac{1}{|\zeta-z|}$. Then the quantity in final bracket of the
above inequality equals to
\begin{eqnarray*}
(1+P)(1+2Q)^2-1&=&P+4Q+4Q^2+4PQ+4PQ^2 \\
&<&9P+8Q\leq\frac{18|c|+8}{R-|z|}.
\end{eqnarray*}
Therefore, combining the above inequalities yields
$$\left(\frac{\partial G_D(z+c,\zeta)}{\partial
{\bf n}}-\frac{\partial G_D(z,\zeta)}{\partial {\bf n}}\right)^+{\rm
d}s\leq\frac{18|c|+8}{R-|z|}\frac{R^2-|z|^2}{|\zeta-z|^2}{\rm
d}\theta.$$ The same argument implies
$$\left(\frac{\partial G_D(z,\zeta)}{\partial
{\bf n}}-\frac{\partial G_D(z+c,\zeta)}{\partial {\bf n}}\right)^+
{\rm d}s\leq \frac{18|c|+8}{R-|z|}\frac{R^2-|z|^2}{|\zeta-z|^2}{\rm
d}\theta.$$
Thus it follows from (\ref{equ2.5}) that
\begin{equation}\label{equ2.6}m_s\left(B,0,\left|\frac{\partial G_D(z+c,\zeta)}{\partial {\bf
n}}-\frac{\partial G_D(z,\zeta)}{\partial {\bf n}}\right|\right){\rm
d}s\leq\frac{18(2|c|+1)}{R-|z|},\end{equation} by noting that
$$\frac{1}{2\pi}\int_0^{2\pi}\frac{R^2-|z|^2}{|\zeta-z|^2}{\rm
d}\theta=\frac{1}{2\pi}\int_{|\zeta|=R}\frac{\partial
G_D(\zeta,z)}{\partial {\bf n}}{\rm d}s=1.$$ And from
(\ref{equ3.6+}) it follows that
\begin{equation}\label{equ3.6+-}N_s(D,c,u)\leq N_s(R+|c|,u_c)\ \text{and}\
N_s(B+c,c,u)\geq N_s(r-2|c|,u_c).\end{equation}

Now we can establish the lemma on the logarithmic differences for
$\delta$-subharmonic functions.

\begin{thm}\label{thm3.2}\ Let $u$ be a $\delta$-subharmonic function on
$\mathbb{C}$. Then for any complex number $c$ and a real number
$\delta$ with $0<\delta<1/2$, we have
$$m_s(r,u_c-u)<436e(1+|c|)\left(\frac{1}{\phi_\varepsilon(r)}\right)^{\delta}T_s(r,u),\ r\not\in E_\varepsilon,$$
where $E_\varepsilon$ is a set in $r$ with finite logarithmic
measure, for $1<\phi_\varepsilon(r)<r$ and $r>1+|c|$. If
(\ref{equ2.2+}) holds for $T=T_s$, then the above inequality holds
where $\phi_\varepsilon(r)$ is assumed to be $\phi(r)$ and
$E_\varepsilon$ is the $E_\delta$ with $\underline{\rm
dens}E_\delta=0$.
\end{thm}

\begin{proof}\ In view of Theorem \ref{thm2.1+},
(\ref{equ2.6}) and (\ref{equ3.6+-}), from the definition of
$m_s(r,u)$ it follows that
\begin{eqnarray}\label{equ2.9}m_s(r,u_c-u)&\leq &\frac{18(2|c|+1)}{R-r}(m_s(R,u)+m_s(R,-u))\nonumber\\
&\
&+(N_s(R+|c|,u_c)-N_s(r-2|c|,u_c))\nonumber\\
&\ &+(N_s(R,-u)-N_s(r,-u)).\end{eqnarray} We want to find a real and
positive function $V$ in $(0,+\infty)$ such that $V(r)\to\infty\
(r\to \infty)$. To this end, solve the equation in $v$
$$\frac{R}{r-2|c|}\frac{v-r+3|c|}{R-v-|c|}=\frac{1}{V(r)},\
v=\frac{(R-r+2|c|)(r-2|c|)}{RV(r)+r-2|c|}-3|c|+r,$$
$$\frac{1}{v-r}=\frac{RV(r)+r-2|c|}{(R-r-|c|)(r-2|c|)-3|c|RV(r)}\leq \frac{R(V(r)+1)}{(R-r-|c|)(r-2|c|)-3|c|RV(r)}.$$

Taking $R=r+\hat{\phi}^{\delta}(r)$ and
$V(r)+1=\hat{\phi}^{\delta/2}(r)$ with $0<\delta<1$,
$6|c|\leq\hat{\phi}^{\delta/2}(r)$, $\hat{\phi}^{\delta}(r)\leq r$
and $\hat{\phi}(r)\to\infty \ (r\to\infty)$, we have $R\leq 2r$,
$v\geq 3|c|+r$ and $(R-r-|c|)(r-2|c|)\geq 6|c|RV(r)$ so that
$$\frac{1}{v-r}\leq\frac{2R(V(r)+1)}{(R-r-|c|)(r-2|c|)}\leq
6\left(\frac{1}{\hat{\phi}(r)}\right)^{\delta/2},\
\frac{1}{V(r)}\leq
2\left(\frac{1}{\hat{\phi}(r)}\right)^{\delta/2}.$$

Set $\mu=\frac{1}{2\pi}(\Delta u)^+$ and $n_s(t,-u)=\mu(B(0,t))$.
Then
\begin{eqnarray*}N_s(v,-u)-N_s(r,-u)&=&\int_r^v\frac{n_s(t,-u)}{t}{\rm
d}t\leq n_s(v,-u)\log\frac{v}{r}\\
&\leq
&\frac{\log\frac{v}{r}}{\log\frac{R}{v}}(N_s(R,-u)-N_s(v,-u))\\
&\leq &\frac{R}{r}\frac{v-r}{R-v}N_s(R,-u).\end{eqnarray*} The same
calculation yields
$$N_s(v+|c|,u_c)-N_s(r-2|c|,u_c)\leq \frac{R}{r-2|c|}\frac{v-r+3|c|}{R-v-|c|}N_s(R,u_c).$$
With $R$ replaced by $v$ in (\ref{equ2.9}), in view of
(\ref{equ2.6-}), it can be seen that
\begin{eqnarray}\label{equ2.10}m_s(r,u_c-u)&\leq &\frac{18(2|c|+1)}{v-r}(m_s(v,u)+m_s(v,-u))\nonumber\\
&\ &+(N_s(v+|c|,u_c)-N_s(r-2|c|,u_c))\nonumber\\
&\ &+(N_s(v,-u)-N_s(r,-u))\nonumber\\
&\leq & \frac{36(2|c|+1)}{v-r}T_s(v,u)+\frac{1}{V(r)}(N_s(R,u_c)+N_s(R,u))\nonumber\\
&\leq &
\frac{36(2|c|+1)}{v-r}T_s(v,u)+\frac{2(1+|c|)}{V(r)}N_s(R,u).
\end{eqnarray}
 Inequality
(\ref{equ2.10}) together with the above inequalities yield
$$m_s(r,u_c-u)\leq
(436|c|+220)\left(\frac{1}{\hat{\phi}(r)}\right)^{\delta/2}T_s(r+\hat{\phi}^{\delta}(r),u).$$

Take $\hat{\phi}(r)=\phi_{T,\varepsilon}(r)$ for $T(r)=T_s(r,u)$.
Applying (\ref{equ3.7-}) to $T_s$ yields $T_s(R,u)\leq
T(r+\phi_\varepsilon(r),u)\leq eT_s(r,u), \ r\not\in E_\varepsilon$
and then we obtain immediately the desired inequality.

We go to the final part of Theorem \ref{thm3.2}. If (\ref{equ2.2+})
holds, then taking $\hat{\phi}(r)=\phi(r)\to\infty \ (r\to\infty)$, in
view of Lemma \ref{lem2.1}, we have
$$T_s(r+\phi^\delta(r),u)\leq eT_s(r,u),\ r\not\in E_\delta,$$
with $\underline{\rm dens}E_\delta=0$. We get the desired
inequality.
\end{proof}

It is clear that Theorem \ref{thm1.3+} and Theorem \ref{thm1.3} are
direct consequences of Theorem~\ref{thm3.2}. We remark that $c$ is
allowed to vary with $0\leq |c|\leq o(\phi^\delta(r))$ with
$0<\delta<1/2$.

It is well-known that for a meromorphic function $f$ of finite
order, we have
$$m\left(r,\frac{f'}{f}\right)=O(\log r)$$
without any exceptional set of $r>1$. But the exceptional set is
necessary for the case of logarithmic differences. Below let us find
such an example of an entire function with finite order.

\begin{exa}
Taking a sequence of positive numbers $\{r_k\}$ with $r_{k+1}\geq
2r_k, r_1>6$ and a sequence of positive integers $\{n_k\}$, we
consider the infinite product
$$f(z)=f(z;r_k,n_k)=\prod_{k=1}^\infty
\left\{1-\left(\frac{z}{r_k}\right)^{n_k}\right\}.$$ We will prove
that for suitable $n_k$, $f$ is of finite order and
$$m\left(r,\frac{f_3}{f}\right)\geq
(1+o(1))T(r,f_3)\ \text{and}\ T(r,f)=o(T(r,f_3))$$ as $r\in
E\to\infty$, where $E$ has finite logarithmic measure and infinite
linear measure.
\end{exa}

Since $r_k\geq 2^k$, it follows that $f$ is an entire
function on the complex plane. Let us estimate
$N\left(r,\frac{1}{f}\right)$, $N\left(r,\frac{1}{f_3}\right)$ and
$T(r,f)$. We can write
$$N\left(r,\frac{1}{f}\right)=\sum\limits_{k=1}^\infty n_k\log^+\frac{r}{r_k}$$
and
$$N\left(r,\frac{1}{f_3}\right)=\sum\limits_{k=1}^\infty \sum\limits_{j=0}^{n_k-1}\log^+\frac{r}{|-3+\omega_k^jr_k|},$$
where $\omega_k=e^{\frac{2\pi}{n_k}i}$. Since
\begin{eqnarray*}|-3+re^{i\theta}|^2&=&r^2-6r\cos \theta+9\leq
r^2-3\sqrt{2}r+9<(r-1)^2,
\end{eqnarray*}
for $-\frac{\pi}{4}\leq\theta\leq\frac{\pi}{4}$ and $r>4$, we have
$$N\left(r,\frac{1}{f_3}\right)>\sum\limits_{k=1}^\infty\frac{1}{4}n_k\log^+\frac{r}{r_k-1}.$$
For $r_s-\frac{1}{2}\leq r\leq r_s$, we have
$$N\left(r,\frac{1}{f}\right)=\sum\limits_{k=1}^{s-1} n_k\log^+\frac{r}{r_k}$$
and
\begin{eqnarray*}N\left(r,\frac{1}{f_3}\right)&>&\frac{1}{4}\sum\limits_{k=1}^sn_k\log^+\frac{r}{r_k-1}\\
&=&\frac{1}{4}\sum\limits_{k=1}^{s-1}
n_k\log^+\frac{r}{r_k-1}+\frac{1}{4}
n_s\log^+\frac{r}{r_s-1}\\
&>&\frac{1}{4}\sum\limits_{k=1}^{s-1} n_k\log^+\frac{r}{r_k-1}+
\frac{n_s}{8r_s}.\end{eqnarray*} Taking $n_s$ such that
$$n_s>4r_s(\log r_s)^2\sum_{k=1}^{s-1}n_k,$$ it can be seen that
$$N\left(r,\frac{1}{f}\right)=o\left(N\left(r,\frac{1}{f_3}\right)\right),\ r_s-\frac{1}{2}\leq r\leq r_s$$
as $s\to\infty.$ Clearly, $E:=\bigcup\limits_{s=1}^\infty
[r_s-1/2,r_s]$ has positive, finite logarithmic measure and infinite
linear measure. Since $f$ is entire, we have
\begin{eqnarray*}T(r,f)&\leq &\log M(r,f)\leq
\sum_{k=1}^\infty\log\left\{1+\left(\frac{r}{r_k}\right)^{n_k}\right\} \\
&=&\left(\sum_{r_k>r}+\sum_{r_k\leq
r}\right)\log\left\{1+\left(\frac{r}{r_k}\right)^{n_k}\right\}
\\
&\leq &\sum_{r_k>r}\left(\frac{r}{r_k}\right)^{n_k}+\sum_{r_k\leq
r}\left\{\log\left(\frac{r}{r_k}\right)^{n_k}+\log\left[1+\left(\frac{r_k}{r}\right)^{n_k}\right]\right\}\\
 &\leq &\sum_{k=0}^\infty\left(\frac{1}{2}\right)^{kn_k}+
\sum_{r_k\leq r}n_k\log\frac{r}{r_k}+\sum_{r_k\leq
r}\left(\frac{r_k}{r}\right)^{n_k}\\
&<&2+N\left(r,\frac{1}{f}\right)+\sum_{r_k\leq
r}\left(\frac{1}{2}\right)^{(k-1)n_k}\\
&\leq& N\left(r,\frac{1}{f}\right)+4,
\end{eqnarray*}
and so
$$T(r, f)=o\left(N\left(r,\frac{1}{f_3}\right)\right)=o(T(r,f_3)),\ r_s-\frac{1}{2}\leq r\leq r_s.$$

By noting that $m(r,f_3)=T(r,f_3)$ and $m(r,f)=T(r,f)$, we have
$$m\left(r,\frac{f_3}{f}\right)\geq
m(r,f_3)-m(r,f)=T(r,f_3)-T(r,f)=(1+o(1))T(r,f_3)$$ as $r\in
E\to\infty$. It is easy to see that $f$ has finite order greater
than or equal to $1$ if $n_s$ is chosen to satisfy
$$\limsup_{s\to\infty}\frac{\log n_s}{\log r_s}<\infty.$$

\section{Clunie Type Theorems}\label{Clunie_sect}

Formulating a discrete analogue of the Painlev\'e property, that can be reliably applied
to distinguish integrable discrete equations from the non-integrable ones, has been an
important topic of study in the field of discrete integrable systems for over two decades.
In the complex analytic setting it is natural to consider discrete equations as difference
equations embedded into the complex plane, and to look at their meromorphic solutions \cite{Ablowitz}.
It turns out that requiring the existence of admissible meromorphic solutions of slow growth is a property
that can effectively single out Painlev\'e type equations for both discrete \cite{HalburdKorhonen,RodRistoTohge}
and delay differential equations \cite{HalburdKorhonen2}. We will consider first a general class of
equations with as weak growth condition for the solutions as possible.

A difference polynomial of a meromorphic function $w$ has the form
$$P(z,\vec{w})=\sum_{\lambda\in I}a_\lambda(z) w(z)^{\lambda_0}w(z+c_{1})^{\lambda_1}...
w(z+c_{n})^{\lambda_n},$$ where
$\vec{w}(z)=(w(z),w(z+c_1),...,w(z+c_n))$,
$I=\{(\lambda_0,\lambda_1,...,\lambda_n):\
\lambda_j\in\mathbb{N}\cup\{0\}\}$ is a finite multi-index set and
all of $c_{j} (1\leq j\leq n)$ are non-zero complex constants and
every $a_\lambda(z)$ is a given meromorphic function. The total
degree, denoted by ${\rm deg}_{\vec{w}}(P)$, of $P(z,\vec{w})$ in
$\vec{w}$ is defined as
$${\rm deg}_{\vec{w}}(P):=\max\left\{\sum_{j=0}^n\lambda_j:\
\lambda=(\lambda_0,\lambda_1,...,\lambda_n)\in I\right\}$$ and the
degree, denoted by ${\rm deg}_{c_j}(P)$, of $P(z,\vec{w})$ in
$w(z+c_j)\ (1\leq j\leq n)$ as
$$\hat{\lambda}_{c_j}={\rm deg}_{c_j}(P):=\max\left\{\lambda_j:\
\lambda=(\lambda_0,\lambda_1,...,\lambda_n)\in I\right\}$$ and
$$\hat{\lambda}_{0}={\rm deg}_{0}(P):=\max\left\{\lambda_0:\
\lambda=(\lambda_0,\lambda_1,...,\lambda_n)\in I\right\}.$$
 Thus
the weight of a difference polynomial $P(z,\vec{w})$ in $\vec{w}$ is
$$\hat{\kappa}(P)=\sum_{j=1}^n\hat{\lambda}_{c_j}$$
and we set
$$\kappa(P):=\max\left\{\sum_{j=1}^n\lambda_j:\
\lambda=(\lambda_0,\lambda_1,...,\lambda_n)\in I\right\}.$$
Obviously, $\hat{\kappa}(P)\geq \kappa(P).$

\begin{thm}\label{thm1.1} If the difference equation
\begin{equation}\label{equ1.7a} P(z,\vec{w})=R(z,w),
\end{equation}
where $P(z,\vec{w})$ is a difference polynomial in $\vec{w}$ with
$P(z,0,x_1,...,x_n)\not=0$ and $R(z,w)$ is a rational function in
$w$, has a non-rational meromorphic solution $w$ with
(\ref{equ1.2+}), then
$${\rm deg}_w(R)\leq \sum_{j=0}^n\hat{\lambda}_{c_j}=\hat{k}(P)+\hat{\lambda}_{0}.$$
\end{thm}

Obtaining more precise results requires a careful singularity analysis near the poles of the solution. In what follows, we investigate the number of poles of meromorphic
solutions of difference equations in terms of the Clunie type theorems. The following lemma may be known.

\begin{lem}\label{lem4.1}\ Let $f$ be a meromorphic function. Assume
that
\begin{eqnarray*}P(z,w)&=&a_n(z)w^n+a_{n-1}(z)w^{n-1}+...+a_1(z)w+a_0(z)\\
Q(z,w)&=&b_m(z)w^m+b_{m-1}(z)w^{m-1}+...+b_1(z)w+b_0(z)\not\equiv
0\end{eqnarray*} with $n={\rm deg}_w(P)$ and $m={\rm deg}_w(Q)$, and
meromorphic coefficients $a_j(z) (0\leq j\leq n)$ and $b_i(z) (0\leq
i\leq m)$, and that they do not have common factors. Set
$$R(z,w)=\frac{P(z,w)}{Q(z,w)}.$$
Then
\begin{equation}\label{equ4.1}
m(r,P(z,f))\leq nm(r,f)+\sum_{j=0}^nm(r,a_j) + n\log 2\end{equation}
and
\begin{eqnarray}\label{equ4.2}
(n-m)m(r,f)&\leq &
m(r,R(z,f))+(n-m)m\left(r,\frac{1}{a_n}\right)\nonumber\\
 &\
&+(n-m)\sum_{j=0}^{n-1}m\left(r,a_j\right)
+\sum_{i=0}^{m}m\left(r,b_i\right)\nonumber\\
&\ & +(2n-2m+1)\log 2+(n-m)\log n+\log (m+1).
\end{eqnarray}
\end{lem}

\begin{proof}\ Let us begin by proving (\ref{equ4.1}).
Set $P_j(z,f)=a_nf^{n-j}(z)+...+a_{j+1}f(z)+a_j, 0\leq j\leq n.$ For
a given $r>0$, we have
\begin{eqnarray*}\ m(r,P(z,f))&\leq &m(r,fP_1(z,f))+m(r,a_0)+\log
2\\
&\leq &m(r,P_1(z,f))+m(r,f)+m(r,a_0)+\log 2\\
&\leq &m(r,P_2(z,f))+2m(r,f)+m(r,a_1)+m(r,a_0)+2\log2\\
&\leq &nm(r,f)+\sum_{j=0}^nm(r,a_j)+n\log 2.
\end{eqnarray*}

To prove (\ref{equ4.2}), set $$E=\left\{\theta:\
|f(re^{i\theta})|>1+\frac{2}{|a_n(re^{i\theta})|}\sum_{j=0}^{n-1}|a_j(re^{i\theta})|\right\}$$
and $$F=\left\{\theta:\ 1\leq |f(re^{i\theta})|\leq
1+\frac{2}{|a_n(re^{i\theta})|}\sum_{j=0}^{n-1}|a_j(re^{i\theta})|\right\}.$$
Here we require that for $\theta\in E\cup F$, $z=re^{i\theta}$ is not a pole of $f$ and neither a pole of
$a_j(0\leq j\leq n)$, nor a zero of $a_n$. Then the
calculation below is reasonable. For $\theta\in E$, it is easily seen that
\begin{eqnarray*}
|P(re^{i\theta},f)|&\geq
&|a_n(re^{i\theta})||f(re^{i\theta})|^n-|f(re^{i\theta})|^{n-1}\sum_{j=0}^{n-1}|a_j(re^{i\theta})|\\
&\geq
&|f(re^{i\theta})|^n\left(|a_n(re^{i\theta})|-\frac{1}{|f(re^{i\theta})|}\sum_{j=0}^{n-1}|a_j(re^{i\theta})|\right)\\
&\geq &\frac{1}{2}|a_n(re^{i\theta})||f(re^{i\theta})|^n,
\end{eqnarray*}
so that for $\theta\in E$,
$$|f(re^{i\theta})|^{n-m}\leq
2\cdot\frac{1}{|a_n(re^{i\theta})|}|R(re^{i\theta},f)|\sum_{i=0}^m|b_i(re^{i\theta})|.$$
From the definition of $m(r,f)$ it follows that
\begin{eqnarray*}\ &\ & (n-m)m(r,f)=(n-m)\frac{1}{2\pi}\int_{E\cup F}\log|f(re^{i\theta})|{\rm
d}\theta\\
&\leq &\frac{1}{2\pi}\int_{E}\log|R(re^{i\theta},f)|{\rm
d}\theta+\frac{1}{2\pi}\int_{E}\log\left|\frac{1}{a_n(re^{i\theta})}\right|{\rm
d}\theta+\log 2 \\
 &\ &+\frac{1}{2\pi}\int_{E}\log
\sum_{i=0}^m|b_i(re^{i\theta})|{\rm d}\theta
+(n-m)\frac{1}{2\pi}\int_{F}\log\left(1+\frac{2}{|a_n(re^{i\theta})|}\sum_{j=0}^{n-1}|a_j(re^{i\theta})|\right){\rm
d}\theta\\
&\leq
&m(r,R(z,f))+(n-m)m\left(r,\frac{1}{a_n}\right)+(n-m)\sum_{j=0}^{n-1}m\left(r,a_j\right)+\sum_{i=0}^{m}m\left(r,b_i\right)\\
&\ & +(2n-2m+1)\log 2+(n-m)\log n+\log (m+1)
\end{eqnarray*}
outside a discrete set of $r>0$ where $r=|z|$ is such that either $f$ or  $a_j(0\leq j\leq n)$ has a pole, or $a_n$ has a zero, on a circle of radius $r$, centred at the origin. But since $m(r,f)$ is continuous, it follows that (\ref{equ4.2}) holds for all $r>0$.  
\end{proof}

In general, it is basically impossible to estimate $m(r,R(z,f))$ in
terms of $m(r,f)$, if $R(z,w)$ is not a polynomial in $w$.

\begin{thm}\label{thm4.1}\ Let $f$ be a meromorphic function and
\begin{equation*}
\begin{split}
P(z,\vec{f})&=\sum_{\lambda\in I}a_\lambda(z) f(z)^{\lambda_0}f(z+c_{1})^{\lambda_1}...
f(z+c_{n})^{\lambda_n},\\
Q(z,\vec{f})&=\sum_{\lambda\in J}b_\lambda(z) f(z)^{\lambda_0}f(z+c_{1})^{\lambda_1}...
f(z+c_{n})^{\lambda_n},
\end{split}
\end{equation*}
where $I\cup
J=\{(\lambda_0,\lambda_1,...,\lambda_n):\
\lambda_j\in\mathbb{N}\cup\{0\}\}$ is a finite multi-index set and
all of $c_{j} (1\leq j\leq n)$ are non-zero complex constants, and
the $a_\lambda(z)$ and $b_\lambda(z)$ are meromorphic coefficients.
Then
\begin{eqnarray}\label{equ4.3}m(r,P(z,\vec{f}))&\leq &{\rm deg}_{\vec{w}}(P)m(r,f)+\bigg(\sum_{\lambda\in
I}|\lambda|\bigg)\sum_{j=1}^nm\left(r,\frac{f(z+c_{j})}{f(z)}\right)\nonumber\\
&\ &+\sum_{\lambda\in I}m(r,a_\lambda)+{\rm deg}_{\vec{w}}(P)\log
2|I|,\end{eqnarray}and if, in addition, $P(z,\vec{f})$ has just one
term of maximal total degree, and $Q(z,\vec{f})$ and $P(z,\vec{f})$
do not have common factors, letting
$R(z,\vec{f})=\frac{P(z,\vec{f})}{Q(z,\vec{f})},$ then
\begin{eqnarray}\label{equ4.4}
&\ &({\rm deg}_{\vec{w}}(P)-{\rm deg}_{\vec{w}}(Q))m(r,f)\\ &\leq&
m(r,R(z,\vec{f}))+{\rm deg}_{\vec{w}}(P)\sum_{j=1}^n\hat{\lambda}_j
m\left(r,\frac{f(z)}{f(z+c_{j})}\right)\nonumber\\
&\ &+\bigg({\rm deg}_{\vec{w}}(P)\sum_{|\lambda|<m,\lambda\in
I}|\lambda|+\sum_{\lambda\in J}|\lambda|\bigg)\sum_{j=1}^n
m\left(r,\frac{f(z+c_j)}{f(z)}\right)\nonumber\\
&\ &+{\rm
deg}_{\vec{w}}(P)m\left(r,\frac{1}{a_{\hat{\lambda}}}\right)+{\rm
deg}_{\vec{w}}(P)\sum_{|\lambda|<m}m(r,a_\lambda) +\sum_{\lambda\in
J}m(r,b_\lambda)+C,
\end{eqnarray}
where $m={\rm deg}_{\vec{w}}(P)$,
$\hat{\lambda}=(\hat{\lambda_0},\hat{\lambda_1},...,\hat{\lambda_n})\in
I$ with $|\hat{\lambda}|=m$ and $C$ is a constant only depending on
${\rm deg}_{\vec{w}}(P),{\rm deg}_{\vec{w}}(Q)$ and $|I|, |J|$ which
denote the numbers of elements of $I$ and $J$ in turn.
\end{thm}

\begin{proof}\ Set $m={\rm deg}_{\vec{w}}(P)$ and
$$c_\lambda(z)=\prod_{j=1}^n\left(\frac{f(z+c_j)}{f(z)}\right)^{\lambda_j},\
\lambda=(\lambda_0,\lambda_1,...,\lambda_n)\in I\cup J.$$ Then we
can rewrite $P(z,\vec{f})$ into
$$P(z,\vec{f})=\sum_{\lambda\in I}a_\lambda c_\lambda f(z)^{|\lambda|}
=\bigg(\sum_{|\lambda|=m}a_\lambda c_\lambda\bigg)
f(z)^{m}+\sum_{|\lambda|<m}a_\lambda c_\lambda f(z)^{|\lambda|},$$
where $|\lambda|=\sum_{j=0}^n\lambda_j$. In terms of Lemma
\ref{lem4.1}, we have
\begin{eqnarray*}m(r,P(z,\vec{f}))&\leq &
mm(r,f)+\sum_{k=0}^mm(r,\sum_{|\lambda|=k}
a_\lambda c_\lambda)+{\rm deg}_{\vec{w}}(P)\log 2\\
&\leq& mm(r,f)+\sum_{k=0}^m\sum_{|\lambda|=k}m(r,
a_\lambda c_\lambda)+\sum_{k=0}^m\log\sum_{|\lambda|=k}1+{\rm deg}_{\vec{w}}(P)\log 2\\
&\leq & mm(r,f)+\sum_{\lambda\in I}m(r,c_\lambda)+\sum_{\lambda\in
I}m(r,a_\lambda)+{\rm deg}_{\vec{w}}(P)\log 2|I|\\
&\leq &mm(r,f)+\sum_{\lambda\in
I}\sum_{j=1}^n\lambda_jm\left(r,\frac{f(z+c_{j})}{f(z)}\right)\\
&\ &+\sum_{\lambda\in I}m(r,a_\lambda)+{\rm deg}_{\vec{w}}(P)\log 2|I|\\
&\leq &mm(r,f)+(\sum_{\lambda\in
I}|\lambda|)\sum_{j=1}^nm\left(r,\frac{f(z+c_{j})}{f(z)}\right)\\
&\ &+\sum_{\lambda\in I}m(r,a_\lambda)+{\rm deg}_{\vec{w}}(P)\log
2|I|.\end{eqnarray*} This is (\ref{equ4.3}). A similar argument
together with (\ref{equ4.2}) yields \begin{eqnarray}\label{equ4.5+}
mm(r,f)&\leq&
m(r,P(z,\vec{f}))+mm\bigg(r,\frac{1}{\sum\limits_{|\lambda|=m}a_\lambda
c_\lambda}\bigg)+m\sum_{|\lambda|<m}m(r,a_\lambda
c_\lambda)+O(1)\nonumber\\
&\leq
&m(r,P(z,\vec{f}))+mm\left(r,\frac{1}{a_{\hat{\lambda}}}\right)+m\sum_{j=1}^n\hat{\lambda}_j
m\left(r,\frac{f(z)}{f(z+c_j)}\right)\nonumber\\
&\
&+m\sum_{|\lambda|<m}|\lambda|\sum_{|\lambda|<m}m\left(r,\frac{f(z+c_j)}{f(z)}\right)+m\sum_{|\lambda|<m}m(r,a_\lambda)+O(1),
\end{eqnarray}
where $\lambda\in I$ and
$\hat{\lambda}=(\hat{\lambda}_0,...,\hat{\lambda}_n)\in I$ with
$|\hat{\lambda}|=m$. Since
$$m(r,P(z,\vec{f})\leq m(r,R(z,\vec{f}))+m(r,Q(z,\vec{f})),$$
applying (\ref{equ4.3}) to $m(r,Q(z,\vec{f})$ and in view of
(\ref{equ4.5+}), we deduce (\ref{equ4.4}).
\end{proof}

Clunie's Theorem in \cite{Clunie} on an estimate of Nevanlinna proximity
of a differential polynomial produced from a certain differential
polynomial equation has proven valuable in the study of value
distribution of meromorphic solutions of non-linear differential
equations. It is natural to study the difference analogues of Clunie
Theorem and their applications in difference equations.  We consider the
difference equation
\begin{equation}\label{equ4.5}U(z,\vec{w})P(z,\vec{w})=Q(z,\vec{w}),\end{equation} where
$U(z,\vec{w}),P(z,\vec{w})$ and $Q(z,\vec{w})$ are three difference
polynomials with meromorphic coefficients. Assume that
\eqref{equ4.5} has an admissible meromorphic solution, i.e., the
coefficients are small functions of it. We want to know the number
of poles of the solution and hence we consider estimate of the
proximity function of the solution from the equation (\ref{equ4.5})
with the help of the analogue of the Clunie Theorem. The case when
$U(z,\vec{w})=w^n$ and ${\rm deg}_{\vec{w}}(Q)\leq n$ was
investigated in \cite{RodRisto}. Moreover, in \cite{LaineYang}, the
case when $U(z,\vec{w})$ contains just one term of maximal total
degree in $w$ and its shifts with ${\rm deg}_{\vec{w}}(Q)\leq {\rm
deg}_{\vec{w}}(U)$ was considered, and in \cite{Risto} the case when
$P(z,\vec{w})$ is homogeneous and $U(z,\vec{w})=U(z,w)$ and
$Q(z,\vec{w})=Q(z,w)$ are polynomial in $w$, which covers all
difference equations (\ref{equ1.1++}) -- (\ref{equ1.4++}) listed in the
introduction. In fact, the difference equations in \cite{Risto}
contain many known Painlev\'e difference equations.

Now we make a more careful discussion and establish a new difference
Clunie type theorem. For a meromorphic function $f$, by $S(r,f)$ we
denote a quantity such that $S(r,f)=o(T(r,f))$ as $r\to\infty$
outside of a set with finite logarithmic measure. For a polynomial
$P(z,\vec{x})$ in $\vec{x}=(x_0,x_1,...,x_n)$, by ${\rm ord}_0(P)$
we denote the multiplicity of zero of $P$ as a function of $x_0$ at
$x_0=0$. If ${\rm ord}_0(P)=0$, then $P(z,\vec{x})$ contains at
least one term without $x_0$. For $U, P$ and $Q$ in (\ref{equ4.5}),
we introduce the notations:
\begin{eqnarray*}d_{\vec{w}}&=&\max\{{\deg}_{\vec{w}}(Q),{\deg}_{\vec{w}}(P)+{\deg}_{\vec{w}}(U)\}-\min\{{\deg}_{\vec{w}}(P),{\rm
ord}_0(Q)\}\\
D_{\vec{w}}&=& d_{\vec{w}}-{\deg}_{\vec{w}}(P)\\
&=&\max\{{\deg}_{\vec{w}}(Q)-{\deg}_{\vec{w}}(P),
{\deg}_{\vec{w}}(U)\}-\min\{{\deg}_{\vec{w}}(P),{\rm
ord}_0(Q)\}\\
\tau_{\vec{w}}&=&d_{\vec{w}}-\hat{\kappa}(P).
\end{eqnarray*}

In the following theorem we consider equation \eqref{equ4.5} in the case where $U(z,\vec{w})$ and $Q(z,\vec{w})$ are specialized as polynomials in $w$ with meromorphic coefficients.

\begin{thm}\label{thm4.3}\ Let $P(z,\vec{w})$ be a homogeneous difference polynomial in $\vec{w}$ with
${\rm ord}_0(P)=0$ and $\hat{\lambda}_0<{\rm deg}_{\vec{w}}(P)$ and
let $U(z,w)$ and $Q(z,w)$ be both polynomials in $w$ without any
common factors. Assume that the difference equation (\ref{equ4.5})
has a transcendental meromorphic solution $w$ such that, for some $\varepsilon >0$, 
$$\limsup_{r\to\infty}\frac{(\log
r)^{1+\varepsilon}\log T(r,w)}{r}=0.$$ Then
\begin{equation}\label{equ4.6+} \hat{\kappa}(P)\geq
\max\{{\rm deg}_w(Q)-\hat{\lambda}_0,{\rm
deg}_{w}(U)-\min\{\hat{\lambda}_0,{\rm ord}_0(Q)\}\}.
\end{equation} Furthermore,
$\hat{\kappa}(P)\geq D_{\vec{w}}$. If $D_{\vec{w}}>0$, then we have
\begin{equation}\label{equ4.7+}
\frac{D_{\vec{w}}}{\hat{\kappa}(P)}T(r,w)\leq N(r,w)+S(r,w);
\end{equation}
if $\tau_{\vec{w}}>0$, then we have
\begin{equation}\label{equ4.7++}
\frac{\tau_{\vec{w}}}{{\rm deg}_{\vec{w}}(P)}T(r,w)\leq
N\left(r,\frac{1}{w}\right)+S(r,w);
\end{equation}
if $\hat{\kappa}(P)= D_{\vec{w}}$, that is, ${\rm
ord}_0(Q)\leq\hat{\lambda}_0$ and
\begin{equation}\label{equ4.8++}\hat{\kappa}(P)={\deg}_w(U)-{\rm ord}_0(Q)\geq{\rm
deg}_w(Q)-\hat{\lambda}_0,\end{equation} then
\begin{equation}\label{equ4.8+}
N(r,w)=T(r,w)+S(r,w),\
N\left(r,\frac{1}{w}\right)=T(r,w)+S(r,w).\end{equation}
\end{thm}

\begin{proof}\ In terms of Theorem \ref{thm4.1} and Theorem \ref{thm1.4}, we have
\begin{equation}\label{equ4.6++++}m\left(r,\frac{P(z,\vec{w})}{w^{{\rm
deg}_{\vec{w}}(P)}}\right)=S(r,w).\end{equation} It follows from
(\ref{equ4.5}) that, in terms of Valiron's Theorem, we have
\begin{eqnarray}\label{equ4.6}
T\left(r,\frac{P(z,\vec{w})}{w^{{\rm
deg}_{\vec{w}}(P)}}\right)&=&T\left(r,\frac{Q(z,w)}{w^{{\rm
deg}_{\vec{w}}(P)}U(z,w)}\right)\nonumber\\
&=&d_{\vec{w}}T(r,w)+S(r,w).
\end{eqnarray}
Combining (\ref{equ4.6++++}) and (\ref{equ4.6}) yields
\begin{equation}\label{equ4.7}N\left(r,\frac{P(z,\vec{w})}{w^{{\rm
deg}_{\vec{w}}(P)}}\right)=d_{\vec{w}}T(r,w)+S(r,w).\end{equation}
On the other hand, by noting that $\kappa(P)={\rm deg}_{\vec{w}}(P)$
for the homogeneous difference polynomial $P(z,\vec{w})$ with ${\rm
ord}_0(P)=0$, we have
\begin{eqnarray}\label{equ4.8}N\left(r,\frac{P(z,\vec{w})}{w^{{\rm
deg}_{\vec{w}}(P)}}\right)&\leq
&N_0\left(r,P(z,\vec{w})\right)+N\left(r,\frac{1}{w^{\kappa(P)}}\right)\nonumber\\
&\leq &\hat{\kappa}(P)N(r+h,w)+S(r,w)+\kappa(P)N\left(r,\frac{1}{w}\right)\nonumber\\
&\leq
&\hat{\kappa}(P)N(r,w)+\kappa(P)N\left(r,\frac{1}{w}\right)+S(r,w),\end{eqnarray}
where $N_0$ means that the poles of only $w$ are not counted and
$h=\max\{|c_j|:1\leq j\leq n\}$. Therefore, we have
\begin{eqnarray}
&&\kappa(P)m\left(r,\frac{1}{w}\right)+(d_{\vec{w}}-\kappa(P))T(r,w) \nonumber \\
&& \quad \leq \hat{\kappa}(P)N(r,w)+S(r,w) \label{equ4.10}
\end{eqnarray}
and
\begin{eqnarray}
&& \hat{\kappa}(P)m\left(r,w\right)+(d_{\vec{w}}-\hat{\kappa}(P))T(r,w) \nonumber \\
&& \quad  \leq \kappa(P)N\left(r,\frac{1}{w}\right)+S(r,w).\label{equ4.10+}
\end{eqnarray}
So we obtain \begin{eqnarray*}\hat{\kappa}(P)&\geq& d_{\vec{w}}-\kappa(P)\\
&=&\max\{{\rm deg}_w(Q)-{\rm deg}_{\vec{w}}(P),{\rm deg}_w(U)\}\\
&\ &-\min\{{\deg}_{\vec{w}}(P),{\rm ord}_0(Q)\},\end{eqnarray*} by
noting that ${\rm deg}_{\vec{w}}(P)=\kappa(P)$. In view of Theorem
\ref{thm1.1}, we have
$$\hat{\kappa}(P)+\hat{\lambda}_0\geq \max\{{\rm
deg}_w(Q),{\rm deg}_w(U)\}.$$ Now,
$$\hat{\kappa}(P)\geq {\rm deg}_w(U)-\min\{{\deg}_{\vec{w}}(P),{\rm
ord}_0(Q)\}$$ and
$$\hat{\kappa}(P)+\hat{\lambda}_0\geq {\rm deg}_w(U),$$
so that \begin{eqnarray*}\hat{\kappa}(P)&\geq& {\rm
deg}_w(U)-\min\{\hat{\lambda}_0,\min\{{\deg}_{\vec{w}}(P),{\rm
ord}_0(Q)\}\}\\
&=&{\rm deg}_w(U)-\min\{\hat{\lambda}_0,{\deg}_{\vec{w}}(P),{\rm
ord}_0(Q)\}\\
&=&{\rm deg}_w(U)-\min\{\hat{\lambda}_0,{\rm
ord}_0(Q)\},\end{eqnarray*} by noting that
$\hat{\lambda}_0\leq{\deg}_{\vec{w}}(P)$. Also
$$\hat{\kappa}(P)\geq{\rm
deg}_w(Q)-\hat{\lambda}_0.$$ These imply (\ref{equ4.6+}).

Consider the case of $d_{\vec{w}}>{\rm deg}_{\vec{w}}(P)$, that is
to say, ${\rm deg}_w(Q)>{\rm deg}_{\vec{w}}(P)+\min\{{\rm
deg}_{\vec{w}}(P),{\rm ord}_0(Q)\}$ or ${\rm deg}_w(U)>\min\{{\rm
deg}_{\vec{w}}(P),{\rm ord}_0(Q)\}$. Then (\ref{equ4.7+}) follows
from (\ref{equ4.10}) and (\ref{equ4.7++}) from
(\ref{equ4.10+}).

Assume that $\hat{\kappa}(P)=D_{\vec{w}}$. In view of
(\ref{equ4.10}), it can be seen that
$${\rm
deg}_{\vec{w}}(P)m\left(r,\frac{1}{w}\right)+\hat{\kappa}(P)m(r,w)=S(r,w).$$
This immediately yields (\ref{equ4.8+}).

Let us make a careful discussion of the condition
$\hat{\kappa}(P)=D_{\vec{w}}$. In terms of (\ref{equ4.6+}), we have
\begin{eqnarray*}\hat{\kappa}(P)&=&\max\{{\rm deg}_w(Q)-\hat{\lambda}_0,{\rm
deg}_w(U)-\min\{\hat{\lambda}_0,{\rm ord}_0(Q)\}\}\\
&=&\max\{{\rm deg}_w(Q)-\hat{\lambda}_0+\min\{\hat{\lambda}_0,{\rm
ord}_0(Q)\},{\rm
deg}_w(U)\}-\min\{\hat{\lambda}_0,{\rm ord}_0(Q)\}\\
&=&\max\{{\deg}_w(Q)-{\deg}_{\vec{w}}(P),{\deg}_w(U)\}-\min\{{\deg}_{\vec{w}}(P),{\rm
ord}_0(Q)\}\end{eqnarray*} so that
\begin{equation}\label{equ4.14}\min\{\hat{\lambda}_0,{\rm
ord}_0(Q)\}\}=\min\{{\deg}_{\vec{w}}(P),{\rm ord}_0(Q)\}
\end{equation}
and $$\max\{{\rm deg}_w(Q)-\hat{\lambda}_0
+\min\{\hat{\lambda}_0,{\rm ord}_0(Q)\},{\rm deg}_w(U)\}$$
\begin{equation}\label{equ4.15}=\max\{{\deg}_w(Q)-{\deg}_{\vec{w}}(P),{\deg}_w(U)\}.
\end{equation}
Since $\hat{\lambda}_0<{\deg}_{\vec{w}}(P)$, from (\ref{equ4.14}) it
follows that ${\rm ord}_0(Q)\leq\hat{\lambda}_0.$ Thus
(\ref{equ4.15}) reduces to $$\max\{{\rm deg}_w(Q)-\hat{\lambda}_0
+{\rm ord}_0(Q),{\rm deg}_w(U)\}$$
\begin{equation}\label{equ4.16}=\max\{{\deg}_w(Q)-{\deg}_{\vec{w}}(P),{\deg}_w(U)\}.
\end{equation}

Suppose that
$${\deg}_w(U)\leq{\deg}_w(Q)-{\deg}_{\vec{w}}(P).$$ Then by noting that
${\rm deg}_w(Q)-\hat{\lambda}_0+{\rm ord}_0(Q)\geq
{\deg}_w(Q)-{\deg}_{\vec{w}}(P)\geq {\deg}_w(U)$, (\ref{equ4.16})
yields
 ${\rm deg}_w(Q)-\hat{\lambda}_0+{\rm
ord}_0(Q)={\deg}_w(Q)-{\deg}_{\vec{w}}(P)$ and so ${\rm ord}_0(Q)=0$
and $\hat{\lambda}_0={\deg}_{\vec{w}}(P)$, and a contradiction is
derived. Therefore, ${\deg}_w(Q)-{\deg}_{\vec{w}}(P)<{\deg}_w(U)$.
It can be seen, by (\ref{equ4.16}), that ${\deg}_w(U)\geq{\rm
deg}_w(Q)-\hat{\lambda}_0+{\rm ord}_0(Q)$.

In conclusion, $\hat{\kappa}(P)=D_{\vec{w}}$ if and only if ${\rm
ord}_0(Q)\leq\hat{\lambda}_0$ and $$\hat{\kappa}(P)={\deg}_w(U)-{\rm
ord}_0(Q)\geq{\rm deg}_w(Q)-\hat{\lambda}_0.$$ These are
(\ref{equ4.8++}).
\end{proof}

We make a remark. If the condition of the growth for $w$ is
(\ref{equ2.2+}), then the results in Theorem \ref{thm4.3} also holds
for $S(r,w)=o(T(r,w))$ outside a subset $E$ of $[1,+\infty)$ with
$\underline{\rm dens}E=0$.

Theorem \ref{thm4.3} is an improvement of the results in
\cite{Risto}, where the conditions for the existence of a
meromorphic solution of the difference equation was not considered.
The main results in \cite{Risto} should be with $\hat{\kappa}(P)$
instead of $\kappa(P)$.

 As an application of Theorem \ref{thm4.3}, we
investigate the difference equation
\begin{equation}\label{equ4.18}
P(z,w,\overline{w},\underline{w})=\overline{w}\underline{w}+\overline{w}w+w\underline{w}=\frac{Q(z,w)}{U(z,w)},
\end{equation}
where $Q(z,w)$ and $U(z,w)$ are both polynomials in $w$ and without
any common factors. Then $\hat{\kappa}(P)=2,\ \hat{\lambda}_0=1,\
{\rm deg}_{\vec{w}}(P)=2$ so that $D_{\vec{w}}=\tau_{\vec{w}}$ and
$$D_{\vec{w}}=
\max\{{\rm deg}_w(Q)-2,{\rm deg}_w(U)\}-\min\{2,{\rm
ord}_0(Q)\}\leq\hat{\kappa}(P)=2.$$

In view of Theorem \ref{thm4.3}, for $D_{\vec{w}}=\tau_{\vec{w}}=1$
we have
$$\frac{1}{2}T(r,w)\leq N(r,w)+S(r,w)\ \text{and}\ \frac{1}{2}T(r,w)\leq N\left(r,\frac{1}{w}\right)+S(r,w);$$
for $D_{\vec{w}}=2$, we have
$$T(r,w)\leq N(r,w)+S(r,w)\ \text{and}\ T(r,w)\leq N\left(r,\frac{1}{w}\right)+S(r,w).$$

In view of (\ref{equ4.6+}) in Theorem \ref{thm4.3}, we have
$$\max\{{\rm deg}_w(Q)-1,{\rm deg}_w(U)-\min\{1,{\rm
ord}_0(Q)\}\}\leq\hat{\kappa}(P)=2.$$ Equivalently
$${\rm deg}_w(Q)\leq 3,\ {\rm deg}_w(U)\leq 2+\min\{1,{\rm
ord}_0(Q)\}.$$ Since ${\rm deg}_w(Q)\leq 3,$ it follows that we only
need to observe
$${\rm
deg}_w(U)-2\leq \min\{1,{\rm ord}_0(Q)\}.$$

Let us divide into three cases to discuss.

\

\noindent{\bf (I)} Assume that ${\rm ord}_0(Q)=0$. Then ${\rm deg}_w(U)\leq
2$ and
$$D_{\vec{w}}=
\max\{{\rm deg}_w(Q)-2,{\rm deg}_w(U)\}$$ so that $0\leq
D_{\vec{w}}\leq 2.$ When $D_{\vec{w}}=0,$ we have ${\rm deg}_w(U)=0$
and ${\rm deg}_w(Q)\leq 2$. When $D_{\vec{w}}=1$, we have ${\rm
deg}_w(U)=1, {\rm deg}_w(Q)\leq 3$ or ${\rm deg}_w(U)=0, {\rm
deg}_w(Q)=3$. When $D_{\vec{w}}=2\ (=\hat{\kappa}(P))$, we have
${\rm deg}_w(U)=2$ and ${\rm deg}_w(Q)\leq 3$. The corresponding
difference equations have the following forms:
\begin{eqnarray}\ \ \ \overline{w}\underline{w}+\overline{w}w+w\underline{w}&=&a_2w^2+a_1w+a_0,
\ a_0\not=0, D_{\vec{w}}=0;\\
\ \ \
\overline{w}\underline{w}+\overline{w}w+w\underline{w}&=&a_3w^3+a_2w^2+a_1w+a_0,
\ a_0\not=0, a_3\not=0, D_{\vec{w}}=1;\\
\ \ \
\overline{w}\underline{w}+\overline{w}w+w\underline{w}&=&\frac{a_3w^3+a_2w^2+a_1w+a_0}{w+b_0},
\ a_0\not=0, D_{\vec{w}}=1;\\
\ \ \
\overline{w}\underline{w}+\overline{w}w+w\underline{w}&=&\frac{a_3w^3+a_2w^2+a_1w+a_0}{w^2+b_1w+b_0},
\ a_0\not=0, D_{\vec{w}}=2.\end{eqnarray}

\

\noindent{\bf (II)} Assume that ${\rm ord}_0(Q)=1$. Then ${\rm deg}_w(U)\leq
3$ and
$$D_{\vec{w}}= \max\{{\rm deg}_w(Q)-2,{\rm deg}_w(U)\}-1$$ so that $-1\leq D_{\vec{w}}\leq 2.$
When $D_{\vec{w}}=0$, we have ${\rm deg}_w(U)=1, {\rm deg}_w(Q)\leq
3$ or ${\rm deg}_w(U)=0, {\rm deg}_w(Q)=3$. When $D_{\vec{w}}=1$, we
have ${\rm deg}_w(U)=2$ and ${\rm deg}_w(Q)\leq 3$. When
$D_{\vec{w}}=2 (=\hat{\kappa}(P))$, we have ${\rm deg}_w(U)=3$ and
${\rm deg}_w(Q)\leq 3$. When $D_{\vec{w}}=-1$, we have ${\rm
deg}_w(U)=0$ and $1\leq {\rm deg}_w(Q)\leq 2$. The corresponding
difference equations have the following forms:
\begin{eqnarray}\ \ \ \overline{w}\underline{w}+\overline{w}w+w\underline{w}&=&\frac{w(a_2w^2+a_1w+a_0)}{w+b_0},
\ a_0\not=0, b_0\not=0, D_{\vec{w}}=0;\\
\ \ \
\overline{w}\underline{w}+\overline{w}w+w\underline{w}&=&w(a_2w^2+a_1w+a_0),
\ a_0\not=0, a_2\not=0, D_{\vec{w}}=0;\\
\ \ \
\overline{w}\underline{w}+\overline{w}w+w\underline{w}&=&\frac{w(a_2w^2+a_1w+a_0)}{w^2+b_1w+b_0},
\ a_0\not=0, b_0\not=0, D_{\vec{w}}=1;\\
\ \ \
\overline{w}\underline{w}+\overline{w}w+w\underline{w}&=&\frac{w(a_2w^2+a_1w+a_0)}{w^3+b_2w^2+b_1w+b_0},
\ a_0\not=0, b_0\not=0, D_{\vec{w}}=2;\\
\ \ \
\overline{w}\underline{w}+\overline{w}w+w\underline{w}&=&w(a_1w+a_0),\
a_0\not=0, D_{\vec{w}}=-1.\end{eqnarray}

\

\noindent{\bf (III)} Assume that $2\leq {\rm ord}_0(Q)\leq 3$ so that $2\leq
{\rm deg}_w(Q)\leq 3$. Then ${\rm deg}_w(U)\leq 3$ and
$$D_{\vec{w}}= \max\{{\rm deg}_w(Q)-2,{\rm deg}_w(U)\}-2$$ so that $-2\leq D_{\vec{w}}\leq 1.$
When $D_{\vec{w}}=0,$ we have ${\rm deg}_w(U)=2$ and $2\leq {\rm
deg}_w(Q)\leq 3$. When $D_{\vec{w}}=1$, we have ${\rm deg}_w(U)=3$
and $2\leq {\rm deg}_w(Q)\leq 3$. When $D_{\vec{w}}=-1$, we have
${\rm deg}_w(U)=1,\ 2\leq {\rm deg}_w(Q)\leq 3$ or ${\rm
deg}_w(U)=0,\ {\rm deg}_w(Q)=3$. When $D_{\vec{w}}=-2$, we have
${\rm deg}_w(U)=0$ and ${\rm deg}_w(Q)=2.$ The corresponding
difference equations have the following forms:
\begin{eqnarray}\overline{w}\underline{w}+\overline{w}w+w\underline{w}&=&\frac{w^2(a_1w+a_0)}{w^2+b_1w+b_0},
\ b_0\not=0, D_{\vec{w}}=0;\\
\ \ \
\overline{w}\underline{w}+\overline{w}w+w\underline{w}&=&\frac{w^2(a_1w+a_0)}{w^3+b_2w^2+b_1w+b_0},
\ b_0\not=0, D_{\vec{w}}=1;\\
\ \ \
\overline{w}\underline{w}+\overline{w}w+w\underline{w}&=&w^2(a_1w+a_0),
\ a_1\not=0, D_{\vec{w}}=-1;\\
\ \ \
\overline{w}\underline{w}+\overline{w}w+w\underline{w}&=&\frac{w^2(a_1w+a_0)}{w+b_0},
\ b_0\not=0, D_{\vec{w}}=-1;\\
\ \ \
\overline{w}\underline{w}+\overline{w}w+w\underline{w}&=&a_2w^2, \
a_2\not=0, D_{\vec{w}}=-2.\end{eqnarray}

Some of the 14 equations listed above can be ruled out according to
the growth of the existing meromorphic solution.

\begin{thm}\label{thm4.3b}\ Let $w$ be a non-rational meromorphic solution of the
difference equation
\begin{equation}\overline{w}\underline{w}+\overline{w}w+w\underline{w}=a_3w^3+a_2w^2+a_1w+a_0,
\ a_3\not=0.\end{equation} Then for a $D>1$ and $K>0$, we have
$$T(r,w)\geq KD^r,\ r>1.$$
\end{thm}

\begin{proof}\ Suppose that $w$ has finitely many poles. For a
sufficiently large $r$, we have
\begin{eqnarray*}
M(r,\overline{w}\underline{w}+\overline{w}w+w\underline{w})&\leq&
M(r,\overline{w}\underline{w})+M(r,\overline{w}w)+M(r,w\underline{w})\\
&\leq& 3M^2(r+1,w),
\end{eqnarray*}
where $M(r,*)$ denotes the maximal modulus of function $*$ on
$|z|=r$. Also
\begin{eqnarray*}
 && M(r,a_3w^3+a_2w^2+a_1w+a_0)\\ && \quad\geq M^3(r,
w)\left(|a_3(z_0)|-\frac{1}{M(r,w)}(|a_2(z_0)|+|a_1(z_0)|+|a_0(z_0)|\right),
\end{eqnarray*}
where $z_0$ with $|z_0|=r$ is such that $M(r,w)=|w(z_0)|.$ So
$$\frac{3}{2}\log M(r,w)\leq \log M(r+1,w)+S(r,w).$$

Now suppose that $w$ has infinitely many poles. When $z_0$ is a pole
of $w$ with order $k$, we write $w(z_0)=\infty^{k_0}$. If
$a_j(z_0)\not=0,\infty (0\leq j\leq 3)$, then
$$(\overline{w}\underline{w}+\overline{w}w+w\underline{w})(z_0)=\infty^{3k_0}.$$
We obtain $\overline{w}(z_0)=\infty^{k_1}$ or
$\underline{w}(z_0)=\infty^{k_1}$ with $k_1\geq \frac{3}{2}k_0$.
Without loss of generality, we assume $w(z_0+1)=\infty^{k_1}$. If
$a_j(z_0+1)\not=0,\infty (0\leq j\leq 3)$, then
$$(\overline{w}\underline{w}+\overline{w}w+w\underline{w})(z_0+1)=\infty^{3k_1}.$$
It can now be seen that $w(z_0+2)=\infty^{k_2}$ with $k_2\geq
\frac{3}{2}k_1$. Proceeding inductively, we have
$w(z_0+n)=\infty^{k_n}$ with $k_n\geq
\left(\frac{3}{2}\right)^nk_0$. Thus, we get
$$\left(\frac{3}{2}\right)^nn(r,w)\leq n(r+n,w)+\sum_{j=0}^3\left(n(r+n,a_j)+n\left(r+n,\frac{1}{a_j}\right)\right).$$
so that we can deduce the inequality for $N(r,w)$, and therefore
$N(r,w)>K\left(\frac{3}{2}\right)^r$ for some positive constant $K$.
\end{proof}

In terms of Theorem \ref{thm4.3b}, (4.25), (4.29) and (4.35) are
ruled out of the above list if we assume the existence of admissible
meromorphic solutions with the growth at most hyper-order $1$ and
minimal hyper-type.

Equations (4.31) and (4.34) can be rewritten into the following form
$$(\overline{w}+w)(\underline{w}+w)=\frac{G(z,w)}{U(z,w)},$$
where ${\rm deg}_w(G)=5,\ {\rm deg}_w(U)=3$. This derives a
contradiction, see the paragraph after (1.4). Therefore, ${\rm
deg}_w(U)\leq 2$ in the all equations listed above. Thus equations
(4.31) and (4.34) are ruled out. It has been proven that the high
density of poles of meromorphic solutions is a key in singling out
the Painlev\'e type difference equations with the form
(\ref{equ4.18}). However, we emphasize that when $D_{\vec{w}}\leq 0$
from (\ref{equ4.7+}), we cannot always avoid the possibility that
the solution $w$ has few poles. Lemma 2.5 in \cite{Zhang1} shows
that if ${\rm deg}_w(Q)=3$ in (\ref{equ4.18}), then we have
\begin{equation}\label{equ4.37+}m(r,w)=T(r,w)+S(r,w).\end{equation} Therefore, a meromorphic
solution $w$ of (4.26) or (4.27) with $a_3\not\equiv 0$, (4.28),
(4.30) or (4.31) with $a_2\not\equiv 0$, (4.33), (4.43) or (4.36)
with $a_1\not\equiv 0$ satisfies (\ref{equ4.37+}). However, from
(\ref{equ4.7+}), if $D_{\vec{w}}>0$, then (\ref{equ4.37+}) does not
hold.

An observation to (4.27) yields that for ${\rm deg}_w(Q)<3$ and
${\rm deg}_w(U)=2$ in (\ref{equ4.18}), (\ref{equ4.37+}) holds
possibly.

Therefore, we have the following

\begin{thm} If the difference equation (\ref{equ4.18}) have a meromorphic
solution with the growth of (\ref{equ2.2+}), then (\ref{equ4.18})
reduces to one of the following equations:

\begin{eqnarray*}\ \ \ \overline{w}\underline{w}+\overline{w}w+w\underline{w}&=&a_2w^2+a_1w+a_0,
\ a_0\not=0, D_{\vec{w}}=0;\nonumber\\
\ \ \
\overline{w}\underline{w}+\overline{w}w+w\underline{w}&=&\frac{a_2w^2+a_1w+a_0}{w+b_0},
\ a_0\not=0, D_{\vec{w}}=1;\nonumber\\
\ \ \
\overline{w}\underline{w}+\overline{w}w+w\underline{w}&=&\frac{a_2w^2+a_1w+a_0}{w^2+b_1w+b_0},
\ a_0\not=0, D_{\vec{w}}=2;\nonumber\\
\ \ \
\overline{w}\underline{w}+\overline{w}w+w\underline{w}&=&\frac{w(a_2w^2+a_1w+a_0)}{w+b_0},
\ a_0\not=0, b_0\not=0, D_{\vec{w}}=0;\nonumber\\
\ \ \
\overline{w}\underline{w}+\overline{w}w+w\underline{w}&=&\frac{w(a_1w+a_0)}{w^2+b_1w+b_0},
\ a_0\not=0, b_0\not=0, D_{\vec{w}}=1;\nonumber\\
\ \ \
\overline{w}\underline{w}+\overline{w}w+w\underline{w}&=&w(a_1w+a_0),\
a_0\not=0, D_{\vec{w}}=-1;\nonumber\\
\ \ \
\overline{w}\underline{w}+\overline{w}w+w\underline{w}&=&\frac{w^2(a_1w+a_0)}{w^2+b_1w+b_0},
\ b_0\not=0, D_{\vec{w}}=0;\nonumber\\
\ \ \
\overline{w}\underline{w}+\overline{w}w+w\underline{w}&=&\frac{w^2(a_1w+a_0)}{w+b_0},
\ b_0\not=0, D_{\vec{w}}=-1;\nonumber\\
\ \ \
\overline{w}\underline{w}+\overline{w}w+w\underline{w}&=&a_2w^2, \
a_2\not=0, D_{\vec{w}}=-2.\end{eqnarray*}
\end{thm}
These above equations have been investigated in Zhang \cite{Zhang},
\cite{Zhang1} and Wen \cite{Wen} to determine the conditions which
their coefficients should satisfy.

\

{\bf Acknowledgements.} We would like to express our gratitude to the
referee for his valuable and careful comments in pointing out
mistakes and improving this paper.

\end{document}